\documentclass[12pt,american]{amsart}
\usepackage[T1]{fontenc}
\usepackage[latin9]{inputenc}
\usepackage{geometry}
\usepackage{amstext}
\usepackage{amsthm}
\usepackage{amssymb}
\usepackage{graphicx}

\makeatletter

\providecommand{\tabularnewline}{\\}

\numberwithin{equation}{section}
\numberwithin{figure}{section}
\numberwithin{table}{section}
\theoremstyle{plain}
\newtheorem{thm}{\protect\theoremname}[section]
\theoremstyle{definition}
\newtheorem{example}[thm]{\protect\examplename}
\theoremstyle{definition}
\newtheorem{defn}[thm]{\protect\definitionname}
\theoremstyle{plain}
\newtheorem{prop}[thm]{\protect\propositionname}
\theoremstyle{remark}
\newtheorem{rem}[thm]{\protect\remarkname}
\theoremstyle{plain}
\newtheorem{cor}[thm]{\protect\corollaryname}
\theoremstyle{plain}
\newtheorem{lem}[thm]{\protect\lemmaname}

\usepackage[fontsize=12pt]{scrextend}
\usepackage{graphicx,cases,mathtools}

\newtheoremstyle{plain}
  {\topsep}   
  {\topsep plus 20pt minus 5pt}   
  {}  
  {0pt}       
  {\bfseries} 
  {.}         
  {5pt plus 1pt minus 1pt} 
  {}          

\usepackage{enumitem}
\setlist{itemsep=0pt,topsep=0pt,parsep=1pt,partopsep=0pt}

\usepackage{pgf,tikz}

\usetikzlibrary{arrows}

\usetikzlibrary{shapes}

\usetikzlibrary{patterns}

\renewcommand*\env@cases[1][1]{%
  \let\@ifnextchar\new@ifnextchar
  \left\lbrace
  \def\arraystretch{#1}%
  \array{@{}l@{\quad}l@{}}%
}

\makeatother

\usepackage{babel}
\providecommand{\corollaryname}{Corollary}
\providecommand{\definitionname}{Definition}
\providecommand{\examplename}{Example}
\providecommand{\lemmaname}{Lemma}
\providecommand{\propositionname}{Proposition}
\providecommand{\remarkname}{Remark}
\providecommand{\theoremname}{Theorem}

\begin{document}
\global\long\def\GEN{\text{GEN}}%

\global\long\def\DNG{\text{DNG}}%

\global\long\def\Dzero{\mathcal{D}_{G,0}}%

\global\long\def\calA{\mathcal{A}}%

\global\long\def\Conv{\text{Conv}}%

\global\long\def\Ex{\text{Ex}}%

\global\long\def\Int{\text{Int}}%

\global\long\def\Opt{\text{Opt}}%

\global\long\def\mex{\text{mex}}%

\global\long\def\nim{\text{nim}}%

\global\long\def\pty{\text{pty}}%

\global\long\def\min{\text{min}}%

\global\long\def\nOpt{\text{nOpt}}%

\global\long\def\type{\text{type}}%

\pagebreak{}
\title{Impartial Achievement Games on Convex Geometries\\
}
\author{Stephanie McCoy, N\'andor Sieben}
\curraddr{Northern Arizona University, Department of Mathematics and Statistics,
Flagstaff, AZ 86011-5717, USA}
\thanks{Date: \the\month/\the\day/\the\year}
\email{scw249@nau.edu}
\email{nandor.sieben@nau.edu}
\begin{abstract}
We study a game where two players take turns selecting points of a
convex geometry until the convex closure of the jointly selected points
contains all the points of a given winning set. The winner of the
game is the last player able to move. We develop a structure theory
for these games and use it to determine the nim number for several
classes of convex geometries, including one-dimensional affine geometries,
vertex geometries of trees, and games with a winning set consisting
of extreme points. 
\end{abstract}

\maketitle

\section{Introduction}

A convex geometry is an abstract generalization of the notion of convexity
on a finite set of points in Euclidean space. We study an achievement
game where two players take turns selecting previously unselected
points of a convex geometry until the convex closure of the jointly
selected points contains all the points of a given winning set. The
winner of the game is the last player able move. That is, the winner
is the first player to make the convex hull of the jointly selected
points a superset of the winning set.

This game is a version of a group generating game introduced by Anderson
and Harary \cite{anderson.harary:achievement} and further developed
in \cite{Barnes,BeneshErnstSiebenSymAlt,BeneshErnstSiebenDNG,BeneshErnstSiebenGeneralizedDihedral,ErnstSieben}.
Our game is played on a different kind of mathematical object. It
is also a generalization since we introduce a winning set that can
be different from the base set of the mathematical object. The key
tool for studying these generating games is structure equivalence
introduced in \cite{ErnstSieben}. Structure equivalence is an equivalence
relation on the game positions that is compatible with the option
structure of the positions. Taking the quotient of the game digraph
by structure equivalence provides significant simplifications. We
develop a structure theory for our generalization. This allows us
to determine the nim number of our games for several classes of convex
geometries, including one-dimensional affine geometries, vertex geometries
of trees, and games with a winning set consisting of extreme points. 

The structure of the paper proceeds as follows. In Section~\ref{sec:Preliminaries}
we recall some basic terminology of impartial games and convex geometries.
In Section~\ref{sec:achievementGame} we describe the convex closure
achievement game in detail and provide some general results. In Section~\ref{sec:structureTheory}
we introduce structure equivalence and structure diagrams. In Section~\ref{sec:extremeCharacterizations}
we determine the nim numbers of games in which the goal set is a subset
of the extreme point set of the ground set. In Section~\ref{sec:treeCharacterization}
we determine the nim numbers of games played on vertex geometries
of trees. This allows us to easily determine the nim numbers of games
played on affine geometries in $\mathbb{R}$ in Section~\ref{sec:affineCharacterization}
since these affine geometries are isomorphic to vertex geometries
of paths. We conclude with some further questions in Section~\ref{sec:FurtherQuestions}.

\section{Preliminaries\label{sec:Preliminaries}}

If $f:X\to Y$ and $A\subseteq X$, then we often use the standard
$f(A):=\{f(a)\mid a\in A\}$ notation for the image of $A$. The cardinality
of a set $A$ is denoted by $|A|$, and we write $\pty(A):=|A|\mod2$
for the \emph{parity of a set}.

\subsection{Impartial games}

We recall the basic terminology of impartial combinatorial games.
Our general references for the subject are \cite{albert2007lessons,SiegelBook}.

An \emph{impartial} \emph{game} is a finite set $\mathcal{P}$ of
\emph{positions} accompanied with a starting position and a collection
$\Opt(P)\subseteq\mathcal{P}$ of \emph{options} for each position
$P$. In every move of the game, the current position $P$ becomes
an option of $P$ chosen by the player to move. Every game must finish
in finitely many steps. In particular, no position can be reached
twice. We can think of it as a finite acyclic digraph with the positions
as vertices, where there is an arrow from a position to every option
of that position. Game play is moving a token from one vertex to another
along the arrows. The game ends when the token reaches a sink of the
graph. The last player to move is the winner of the game. 

The \emph{minimum excludant} $\mex(A)$ of a set $A$ of non-negative
integers is the smallest non-negative integer that is not in $A$.
The \emph{nim number} $\nim(P)$ of a position $P$ of a game is defined
recursively as the minimum excludant of the nim numbers of the options
of $P$. That is,
\[
\nim(P):=\mex(\nim(\Opt(P))).
\]
The nim number of the game is the nim number of the starting position.

A position is \emph{terminal} if it has no options. A terminal position
$P$ has nim number $\nim(P)=\mex(\nim(\emptyset))=\mex(\emptyset)=0$.
A position $P$ is losing for the player about to move ($P$-position)
if $\nim(P)=0$ and winning ($N$-position) otherwise. The winning
strategy is to always move to an option with nim number $0$ if available.
This places the opponent into a losing position. The nim number is
a central object of interest for impartial games. In addition to determining
the outcome of the game, it also makes it easy to compute the nim
number of sums of games.

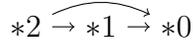
\begin{figure}
\begin{tikzpicture}
\node (p2) at (0,0) {$*2$};
\node (p1) at (1,0) {$*1$};
\node (p0) at (2,0) {$*0$};
\draw[->] (p2) -- (p1);
\draw[->] (p1) -- (p0);
\draw[->] (p2) to[out=25,in=155] (p0);
\end{tikzpicture}

\caption{\label{fig:star2}The acyclic digraph representing the nimber $*2$.}

\end{figure}

\begin{example}
The \emph{nimber} $*n$ is the game with options $\Opt(*n):=\{*0,*1,\ldots,*(n-1)\}$.
The only terminal position of this game is the position $*0$. Induction
shows that $\nim(*n)=n$ since
\[
\begin{aligned}\nim(*n) & =\mex(\nim(\Opt(P)))=\mex(\nim(\{*0,\ldots,*(n-1)\}))\\
 & =\mex(\{\nim(*0),\ldots,\nim(*(n-1))\})=\mex\{0,\ldots,n-1\}=n.
\end{aligned}
\]

The acyclic digraph representation of $*2$ is shown in Figure~\ref{fig:star2}.
In this game the first player can move into position $*0$ and win.
If the first player moves into position $*1$ instead, then the second
player moves into position $*0$ and win.
\end{example}

\subsection{Convex geometries}

We recall some facts about convex geometries from \cite{BjornerZiegler,EJ,KorteLovaszSchrader}. 
\begin{defn}
A \emph{convex geometry} is a pair $(S,\mathcal{K})$, where $S$
is a finite set and $\mathcal{K}$ is a family of subsets of $S$
satisfying the following properties:
\begin{enumerate}
\item $S\in\mathcal{K}$;
\item $K,L\in\mathcal{K}$ implies $K\cap L\in\mathcal{K}$;
\item $S\ne K\in\mathcal{K}$ implies $K\cup\{a\}\in\mathcal{K}$ for some
$a\in S\setminus K$.
\end{enumerate}
\noindent The sets in $\mathcal{K}$ are called \emph{convex}.
\end{defn}

The third condition is called \emph{accessibility}. Following \cite{BjornerZiegler,Dietrich},
we do not require that the empty set and singleton sets are convex.
Note that the family of complements of the convex sets in a convex
geometry forms an \emph{anti-matroid}.

A convex geometry determines a \emph{convex closure} operator $\tau:2^{S}\to2^{S}$
defined by 
\[
\tau(A):=\bigcap\{K\in\mathcal{K}\mid A\subseteq K\}.
\]
A point $a$ of a subset $A$ of a convex geometry is called an \emph{extreme
point} of $A$ if $a\not\in\tau(A\setminus\{a\})$. The set of extreme
points of $A$ is denoted by $\Ex(A)$. 
\begin{prop}
The convex closure operator of a convex geometry satisfies the following
properties:
\begin{enumerate}
\item $A\subseteq\tau(A)$;
\item $A\subseteq B$ implies $\tau(A)\subseteq\tau(B)$;
\item $\tau(\tau(A))=\tau(A)$;
\item $a,b\not\in\tau(A)$ and $a\ne b\in\tau(A\cup\{a\})$ implies $a\not\in\tau(A\cup\{b\})$;
\item $A$ is convex if and only if $\tau(A)=A$;
\item $\tau(A)=\tau(\Ex(A))$.
\end{enumerate}
\end{prop}

The fourth property is called the \emph{anti-exchange property}. So
$\tau$ is a closure operator satisfying the anti-exchange property.
Note that the convex closure of the empty set might not be the empty
set since we do not require the empty set to be convex.

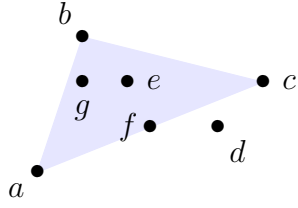
\begin{figure}
\begin{tikzpicture}[scale=.6]
\draw[color=blue!10,fill=blue!10] (0,0) -- (1,3) -- (5,2) -- cycle; 
\node[inner sep=-10,label=200:$a$] (b1) at (0,0) {$\bullet$}; 
\node[inner sep=-10,label=100:$b$] (b3) at (1,3) {$\bullet$}; 
\node[inner sep=0,label=0:$c$] (b4) at (5,2) {$\bullet$}; 
\node[color=white,inner sep=0] at (2,2) {$\bullet$};
\node[inner sep=0,label=0:$e$] (b5) at (2,2) {$\bullet$};
\node[color=white,inner sep=0] at (1,2) {$\bullet$};
\node[inner sep=0,label=270:$g$] (b6) at (1,2) {$\bullet$};
\node[color=white,inner sep=0] (b7) at (2.5,1) {$\bullet$};
\node[inner sep=-10,label=182:$f$] (b7) at (2.5,1) {$\bullet$};
\node[inner sep=-20,label=290:$d$] (b8) at (4,1) {$\bullet$};
\end{tikzpicture}

\caption{\label{fig:hullVertexNonvertex}A point set in $\mathbb{R}^{2}$.
The points in the shaded region belong to $\tau\{a,b,c,e\}$ in the
affine convex geometry.}
\end{figure}

\begin{example}
Let $S$ be a finite subset of $\mathbb{R}^{n}$. The convex subsets
of $S$ are the intersections of $S$ with convex subsets of $\mathbb{R}^{n}$.
The collection $\mathcal{K}$ of convex subsets of $S$ forms a convex
geometry on $S$ called the \emph{affine convex geometry}. The convex
closure operator is $\tau(A)=S\cap\Conv(A)$, where $\Conv(A)$ is
the convex hull of $A$.
\end{example}

\begin{example}
Figure~\ref{fig:hullVertexNonvertex} shows an example of a point
set $S=\{a,b,\ldots,f\}$ in $\mathbb{R}^{2}$. The set of extreme
points of $S$ in the affine convex geometry is $\Ex(S)=\{a,b,c,d\}$.
The set of extreme points of $A=\{a,b,c,e,f,g\}$ is $\Ex(A)=\{a,b,c\}$,
so $A=\tau(\{a,b,c\})$. We also have $A=\tau(\{a,b,c,e\})=\tau(\{a,b,c,f\})$. 
\end{example}

We define deletions by generalizing the definition in \cite{ConvBeta}
and the notion of minor in \cite{affineRep}.
\begin{defn}
Let $(S,\mathcal{K})$ be a convex geometry and $D$ be a subset of
$S$. The \emph{deletion} of $\mathcal{K}$ by $D$ is the collection
\[
\mathcal{K}\setminus D:=\{K\subseteq S\setminus D\mid K\cup D\in\mathcal{K}\}.
\]
\end{defn}

\begin{prop}
If $(S,\mathcal{K})$ is a convex geometry and $D\subseteq S$, then
$(S\setminus D,\mathcal{K}\setminus D)$ is a convex geometry. 
\end{prop}

\begin{proof}
We see that $S\setminus D\in\mathcal{K}\setminus D$ since $(S\setminus D)\cup D=S\in\mathcal{K}$.
If $K,L\in\mathcal{K}\setminus D$ then $K\cap L\in\mathcal{K}\setminus D$
since $(K\cap L)\cup D=(K\cup D)\cap(L\cup D)\in\mathcal{K}$. Now
assume $S\setminus D\ne K\in\mathcal{K}\setminus D$. Since $S\ne K\cup D\in\mathcal{K}$,
accessibility implies that $(K\cup D)\cup\{a\}\in\mathcal{K}$ for
some $a\in S\setminus(K\cup D)$. So $a\in(S\setminus D)\setminus K$
and $K\cup\{a\}\in\mathcal{K}\setminus D$. 
\end{proof}
Note that we do not require $D$ to be convex. This is not necessary
since the empty set does not need to be convex.

\begin{figure}
\begin{tabular}{ccccccccc}
 &  & $\circ\bullet\bullet\,\circ$ & ~ & $\circ\bullet\circ\,\bullet$ & ~ & $\circ\bullet\bullet$ & ~ & $\bullet\,\bullet$\tabularnewline
$\mathcal{K}$ &  & $\{\{0,1\}\}$ &  & $\{\{0\},\{0,1\}\}$ &  & $\{\emptyset,\{0\},\{0,1\}\}$ &  & $\{\emptyset,\{0\},\{1\},\{0,1\}\}$\tabularnewline
$D$ &  & $\{-1,2\}$ &  & $\{-1/2,1/2\}$ &  & $\{-1\}$ &  & $\emptyset$\tabularnewline
$\Ex(S)$ &  & $\emptyset$ &  & $\{1\}$ &  & $\{1\}$ &  & $\{0,1\}$\tabularnewline
\end{tabular}

\caption{\label{fig:convexGeom2}The convex geometries up to isomorphism on
a two-element set $S=\{0,1\}$ represented as deletions of affine
convex geometries on a point set $T$ in $\mathbb{R}$. The points
of $S$ are shown as bullets, while the deleted points are shown as
empty circles.}
\end{figure}

\begin{example}
The convex geometries on a two-element set $S=\{0,1\}$ up to isomorphism
are shown in Figure~\ref{fig:convexGeom2}. Each of these convex
geometries can be represented as a deletion of an affine convex geometry
on a subset $T=S\cup D$ of $\mathbb{R}$. The points in $S$ are
shown with bullets. The deleted points in $D$ are shown as empty
circles. The last row of the table shows the extreme points of $S$.
\end{example}

\begin{rem}
\label{rem:repConj}The main result of \cite{affineRep} is that every
convex geometry that contains the empty set can be represented as
a deletion of an affine convex geometry by a convex set. We conjecture
that a version of this representation result holds even for convex
geometries in which the empty set is not closed. In this conjectured
representation, the deleted set does not need to be convex.
\end{rem}

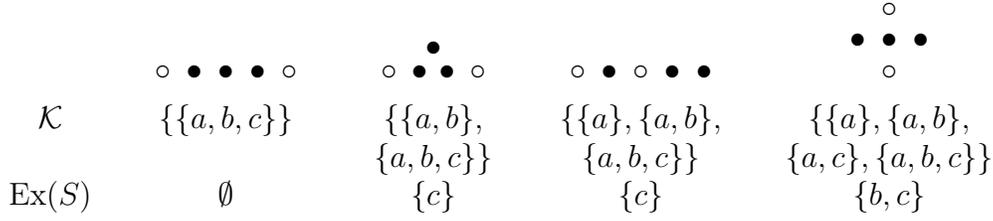
\begin{figure}
\begin{tabular}{ccccccccc}
 &  & \begin{tikzpicture}[scale=.6]
\node at (.7,0) {$\bullet$};
\node at (1.4,0) {$\bullet$};
\node at (2.1,0) {$\bullet$};
\node at (0,0) {$\circ$};
\node at (2.8,0) {$\circ$};
\end{tikzpicture} & ~ & \begin{tikzpicture}[scale=.3,rotate=-30]
\node at (0:.7) {$\bullet$};
\node at (120:.7) {$\bullet$};
\node at (-120:.7) {$\bullet$};
\node at (20:2) {$\circ$};
\node at (-140:2) {$\circ$};
\end{tikzpicture} & ~ & \begin{tikzpicture}[scale=.3,rotate=0]
\node at (0:0) {$\circ$};
\node at (0:-1.4) {$\bullet$};
\node at (0:1.4) {$\bullet$};
\node at (0:2.8) {$\bullet$};
\node at (0:-2.8) {$\circ$};
\end{tikzpicture} & ~ & \begin{tikzpicture}[scale=.3,rotate=90]
\node at (0:0) {$\bullet$};
\node at (0:-1.4) {$\circ$};
\node at (0:1.4) {$\circ$};
\node at (90:1.4) {$\bullet$};
\node at (-90:1.4) {$\bullet$};
\end{tikzpicture}\tabularnewline
$\mathcal{K}$ &  & $\{\{a,b,c\}\}$ &  & $\{\{a,b\},$ &  & $\{\{a\},\{a,b\},$ &  & $\{\{a\},\{a,b\},$\tabularnewline
 &  &  &  & $\{a,b,c\}\}$ &  & $\{a,b,c\}\}$ &  & $\{a,c\},\{a,b,c\}\}$\tabularnewline
$\Ex(S)$ &  & $\emptyset$ &  & $\{c\}$ &  & $\{c\}$ &  & $\{b,c\}$\tabularnewline
\end{tabular}

\caption{\label{fig:convexGeom3}Convex geometries up to isomorphism on a three-element
set $S=\{a,b,c\}$ represented as deletions of affine convex geometries
on a point set $T$ in $\mathbb{R}^{2}$. The points of $S$ are shown
as bullets, while the deleted points are shown as empty circles.}
\end{figure}

\begin{example}
There are four convex geometries up to isomorphism with three points
in which the empty set is not convex. They can be represented as deletions
of affine convex geometries on a point set $T=S\cup D$ of $\mathbb{R}^{2}$,
as shown in Figure~\ref{fig:convexGeom3}. For example, the second
convex geometry in the table can be represented as $S=\{a,b,c\}$
with $a=(1,0)$, $b=(2,0)$, $c=(1.5,1)$ and $D=\{x,y\}$ with $x=(0,0)$,
$y=(3,0)$.
\end{example}

\section{Convex closure achievement game\label{sec:achievementGame}}

We now provide a detailed description of the impartial convex closure
achievement game $\GEN(S,W)$ played on a convex geometry $(S,\mathcal{K})$
with a nonempty winning subset $W$ of $S$. In this game, two players
take turns selecting previously unselected elements of $S$. In turn
$k$ the next player picks $p_{k}$ from $S\setminus\{p_{1},\ldots,p_{k-1}\}$.
The game ends when the convex closure $\tau(P)$ of the jointly selected
points $P=\{p_{1},\ldots,p_{k}\}$ contains $W$. The last player
to make a move wins $\GEN(S,W)$. We call $P$ the current position
of the game. The set of options for a nonterminal position $P$ is
$\Opt(P)=\{P\cup\{s\}\mid s\in S\setminus P\}$.

An important special case is when $W=S$. For this game we use the
notation $\GEN(S)$ instead of the more precise $\GEN(S,S)$.
\begin{example}
\label{exa:3InLine}Consider $\GEN(S)$ on the affine convex geometry
with $S:=\{-1,0,1\}\subseteq\mathbb{R}$ shown in Figure~\ref{fig:TwoSamples}.
If the first player selects $-1$, then the second player can select
$1$ and win the game because $S=\tau\{-1,1\}$. So this first move
is a mistake for the first player.

If the first player selects 0, then without loss of generality we
can assume the second player selects $1$. Then the first player can
win by selecting $-1$. So with best play the first player can win
this game.

Figure~\ref{fig:3InLine} shows the full game digraph of $\GEN(S)$.
The nimbers $*0,*1,*2$ below the sets of chosen points indicate the
nim numbers of the positions.
\end{example}

\begin{figure}
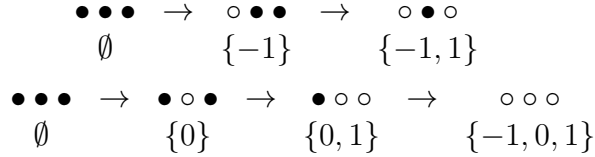

\begin{tabular}{ccccccc}
$\bullet\bullet\bullet$ & $\to$ & $\circ\bullet\bullet$ & $\to$ & $\circ\bullet\circ$ &  & \tabularnewline
$\emptyset$ &  & $\{-1\}$ &  & $\{-1,1\}$ &  & \tabularnewline
\end{tabular}

~

\begin{tabular}{ccccccc}
$\bullet\bullet\bullet$ & $\to$ & $\bullet\circ\bullet$ & $\to$ & $\bullet\circ\circ$ & $\to$ & $\circ\circ\circ$\tabularnewline
$\emptyset$ &  & $\{0\}$ &  & $\{0,1\}$ &  & $\{-1,0,1\}$\tabularnewline
\end{tabular}

\caption{\label{fig:TwoSamples}Two samples of game play for $\protect\GEN(S)$
with $S=\{-1,0,1\}$.}
\end{figure}

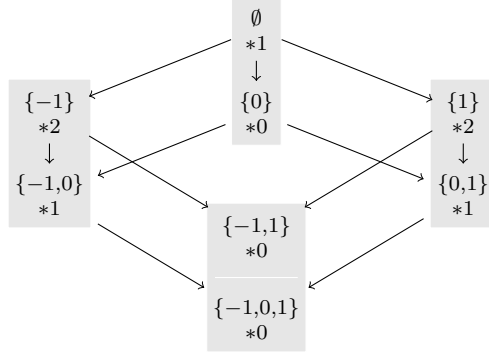
\begin{figure}
\begin{tikzpicture}[scale=1.1]
\draw[draw=white,fill=black!10] (-0.3,0.4) rectangle (.3,-1.4);
\draw[draw=white,fill=black!10] (-0.6,-2.1) rectangle (0.6,-3.9);
\draw[draw=white,fill=black!10] (-3,-0.6) rectangle (-2,-2.4);
\draw[draw=white,fill=black!10] (2.9,-0.6) rectangle (2.1,-2.4);
\draw[draw=white] (-.5,-3) rectangle (.5,-3);
\node (p0) at (0,0) {$\emptyset \atop {*1}$};
\node (p1) at (0,-1) {$\{0\} \atop {*0}$};
\node (p5) at (-2.5,-1) {$\{-1\} \atop {*2}$};
\node (p6) at (2.5,-1) {$\{1\} \atop {*2}$};
\node (p2) at (-2.5,-2) {$\{-1,0\} \atop {*1}$};
\node (p3) at (0,-3.5) {$\{-1,0,1\} \atop {*0}$};
\node (p4) at (2.5,-2) {$\{0,1\} \atop {*1}$};-
\node (p7) at (0,-2.5) {$\{-1,1\} \atop {*0}$};
\draw[->] (p0) -- (p1); 
\draw[->] (p1)-- (p2);
\draw[->] (p2) -- (p3);
\draw[->] (p1) -- (p4);
\draw[->] (p4) -- (p3);
\draw[->] (p0) -- (p5);
\draw[->] (p5) -- (p2);
\draw[->] (p0) -- (p6);
\draw[->] (p6) -- (p4);
\draw[->] (p5) -- (p7);
\draw[->] (p6) -- (p7);
\end{tikzpicture}

\caption{\label{fig:3InLine}Game digraph of $\protect\GEN(S)$ with $S=\{-1,0,1\}$.}
\end{figure}

\begin{prop}
\label{prop:WConvWVerW}The nim numbers of $\GEN(S,\Ex(W))$, $\GEN(S,W)$,
and $\GEN(S,\tau(W))$ are the same.
\end{prop}

\begin{proof}
Let $P\subseteq S$. Since $\Ex(W)\subseteq W\subseteq\tau(W)$, $\tau(W)\subseteq\tau(P)$
implies $\Ex(W)\subseteq\tau(P)$. On the other hand, $\Ex(W)\subseteq\tau(P)$
implies $\tau(W)=\tau(\Ex(W))\subseteq\tau(\tau(P))=\tau(P)$. Hence
the positions in all three games are exactly the same.
\end{proof}
The following is an immediate consequence.
\begin{cor}
The nim numbers of $\GEN(S)$ and $\GEN(S,\Ex(S))$ are the same.
\end{cor}

\section{\label{sec:structureTheory}Structure theory}

\subsection{Structure equivalence}

Structure equivalence is an equivalence relation on the set of game
positions. This relation is compatible with the option structure and
hence the nim numbers of the positions. This allows us to use a smaller
quotient of the game digraph to compute nim numbers.
\begin{defn}
Consider the achievement game $\GEN(S,W)$ and let $M\subseteq S$.
If $W\subseteq\tau(M)$ then $M$ is called a \emph{generating} \emph{set.
}Otherwise, $M$ is called a \emph{non-generating set.} If $M$ is
non-generating and $N$ is generating for all $M\subset N\subseteq S$,
then $M$ is called \emph{maximally non-generating.} We denote the
set of maximally non-generating subsets by $\mathcal{M}$.
\end{defn}

\begin{example}
Consider the affine convex geometry on $S=\{0,1,2,3\}\subseteq\mathbb{R}$
and let $W=\{1,2\}$. Then $\mathcal{M}=\{\{0,1\},\{2,3\}\}$. 
\end{example}

\begin{prop}
A maximally non-generating set $M$ is convex.
\end{prop}

\begin{proof}
The closure $\tau(M)$ of $M$ is a non-generating superset of $M$
since $W\not\subseteq\tau(M)=\tau(\tau(M))$ and $M\subseteq\tau(M)$.
Hence $M=\tau(M)$ since $M$ is maximally non-generating. 
\end{proof}
\begin{defn}
We let 
\[
\mathcal{I}:=\{\bigcap\mathcal{N}\mid\mathcal{N}\subseteq\mathcal{M}\}
\]
be the set of \emph{intersection subsets}. The smallest intersection
subset is the \emph{Frattini subset} $\Phi$, which is the intersection
of all maximally non-generating subsets. 
\end{defn}

Note that with $\mathcal{N}=\emptyset\subseteq\mathcal{M}$ we get
$S=\bigcap\emptyset=\bigcap\mathcal{N}\in\mathcal{I}$. So $S$ is
always an intersection subset.
\begin{example}
Consider the convex geometry $\mathcal{K}=\{\{0\},\{0,1\}\}$ on $S=\{0,1\}$.
If $W=\{0\}$ then $\mathcal{M}=\emptyset$, $\mathcal{I}=\{S\}$,
$\Phi=S$, and $\nim(\GEN(S,W))=0$. Note that in this game there
is only one game position $\emptyset$, and this position is both
a starting and a terminal position. If $W=\{1\}$ then $\mathcal{M}=\{\{0\}\}$,
$\mathcal{I}=\{\{0\},S\}$, and $\Phi=\{0\}$.
\end{example}

\begin{example}
Consider the convex geometry $\mathcal{K}=\{\emptyset,\{0\}\}$ on
$S=\{0\}$. If $W=\{0\}$ then $\mathcal{M}=\{\emptyset\}$, $\mathcal{I}=\{\emptyset,S\}$,
and $\Phi=\emptyset$.
\end{example}

\begin{defn}
For a position $P$ of $\GEN(S,W)$ let
\[
\mathcal{M}_{P}:=\{M\in\mathcal{M}\mid P\subseteq M\},\qquad\lceil P\rceil:=\bigcap\mathcal{M}_{P}.
\]
Two game positions $P$ and $Q$ are \emph{structure equivalent} if
$\lceil P$$\rceil=\lceil Q\rceil$. The \emph{structure class} $X_{I}$
of $I\in\mathcal{I}$ is the equivalence class of $I$ under this
equivalence relation.
\end{defn}

Note that $\lceil I\rceil=I$ for all $I\in\mathcal{I}$ and that
$I\mapsto X_{I}$ is a bijection from $\mathcal{I}$ to the set of
structure classes. 
\begin{example}
\label{exa:3InLineB}Let $W=S=\{-1,0,1\}$ as in Example~\ref{exa:3InLine}.
Then
\[
\begin{aligned}\mathcal{M} & =\{\{-1,0\},\{0,1\}\},\\
\mathcal{I} & =\{\{0\},\{-1,0\},\{0,1\},S\},
\end{aligned}
\]
and $\Phi=\{0\}$. We also have 
\[
\begin{aligned}X_{\{0\}} & =\{\emptyset,\{0\}\}, & X_{\{-1,0\}} & =\{\{-1\},\{-1,0\}\}\\
X_{\{0,1\}} & =\{\{1\},\{0,1\}\}, & X_{S} & =\{\{-1,1\},\{-1,0,1\}\},
\end{aligned}
\]
indicated with the gray boxes in Figure~\ref{fig:3InLine}. So $\lceil\{-1\}\rceil=\{-1,0\}$
and $\lceil\emptyset\rceil=\{0\}$ . Notice the lack of arrow from
$\{-1,1\}$ to $\{-1,0,1\}$.

\end{example}

\begin{prop}
If $\lceil P\rceil=\lceil Q\rceil$ then $\mathcal{M}_{P}=\mathcal{M}_{Q}$.
\end{prop}

\begin{proof}
For a contradiction, assume $\mathcal{M}_{P}\neq\mathcal{M}_{Q}$.
Then without loss of generality, there exists an $M\in\mathcal{M}_{P}\setminus\mathcal{M}_{Q}$.
That is, $M\in\mathcal{M}$ such that $P\subseteq M$ but $Q\not\subseteq M$.
Then there exists $q\in Q\setminus M$. This is impossible since $q\in Q\subseteq\lceil Q\rceil=\lceil P\rceil=\bigcap\mathcal{M}_{P}\subseteq M$.
\end{proof}
\begin{prop}
\label{prop:optStruct}Let $P,Q\in X_{I}\neq X_{J}$. If $\Opt(P)\cap X_{J}\neq\emptyset$
then $\Opt(Q)\cap X_{J}\neq\emptyset$.
\end{prop}

\begin{proof}
Assume $\Opt(P)\cap X_{J}\neq\emptyset$. Then there is an $r\in S\setminus P$
such that $P\cup\{r\}\in X_{J}$. We show that $Q\cup\{r\}\in X_{J}$.
We have 
\[
\begin{aligned}\mathcal{M}_{P\cup\{r\}} & =\{M\in\mathcal{M}\mid P\cup\{r\}\subseteq M\}\\
 & =\{M\in\mathcal{M}\mid P\subseteq M\text{ and }r\in M\}\\
 & =\{M\in\mathcal{M}_{P}\mid r\in M\}
\end{aligned}
\]
and similarly $\mathcal{M}_{Q\cup\{r\}}=\{M\in\mathcal{M}_{Q}\mid r\in M\}$.
Hence
\[
\begin{aligned}\lceil Q\cup\{r\}\rceil & =\bigcap\mathcal{M}_{Q\cup\{r\}}\\
 & =\bigcap\{M\in\mathcal{M}_{Q}\mid r\in M\}\\
 & =\bigcap\{M\in\mathcal{M}_{P}\mid r\in M\}\\
 & =\bigcap\mathcal{M}_{P\cup\{r\}}\\
 & =\lceil P\cup\{r\}\rceil=J
\end{aligned}
\]
since $\mathcal{M}_{P}=\mathcal{M}_{Q}$.
\end{proof}
\begin{defn}
We say $X_{J}$ is an \emph{option} of $X_{I}$ if $X_{J}\cap\Opt(I)\ne\emptyset$.
The set of options of $X_{I}$ is denoted by $\Opt(X_{I})$.
\end{defn}

The following Lemma was proved in \cite{ErnstSieben}.
\begin{lem}
\label{lem:mex}If $A$ and $B$ are sets containing non-negative
integers such that $\mex(A)\in B$, then $\mex(A\cup\{\mex(B)\})=\mex(A)$.
\end{lem}

\begin{prop}
\label{prop:mainThm}If $P,Q\in X_{I}$ and $\pty(P)=\pty(Q)$, then
$\nim(P)=\nim(Q)$.
\end{prop}

\begin{proof}
Let
\[
Z:=\{(P,Q)\mid\lceil P\rceil=\lceil Q\rceil\text{ and }\pty(P)=\pty(Q)\}.
\]
We say $(P,Q)\succeq(M,N)$ when $P\subseteq M$ and $Q\subseteq N$.
Then $(Z,\succeq)$ is a partially ordered set with minimum element
$(S,S)$. We proceed by structural induction on $Z$. Let $(P,Q)\in Z$.
The statement clearly holds if $P=Q$. In particular, it holds if
$(P,Q)=(S,S)$. Otherwise we let $I:=\lceil P\rceil=\lceil Q\rceil$
and consider several cases.

First, assume $P\neq I\neq Q$. Then both $P$ and $Q$ have options
in $X_{I}$. In fact, $P\cup\{s\}\in\Opt(P)\cap X_{I}$ for each $s\in I\setminus P$.
If $M$ and $N$ are options of $P$ and $Q$ in $X_{I}$ respectively,
then $\pty(M)=\pty(N)$. Hence $\nim(M)=\nim(N)$ by induction since
$(P,Q)\succ(M,N)$ in $Z$. If $P$ has an option $M$ in some $X_{J}\neq X_{I}$,
then $Q$ also has an option $N$ in $X_{J}$ by Proposition~\ref{prop:optStruct}.
Since $\lceil M\rceil=J=\lceil N\rceil$ and $\pty(M)=\pty(N)$, we
have $(P,Q)\succ(M,N)$. Hence $\nim(M)=\nim(N)$ by induction. This
proves that $\nim(\Opt(P))=\nim(\Opt(Q))$. Thus $\nim(P)=\nim(Q)$.

Now, assume $P\ne I=Q$. In this case $Q$ does not have any options
in $X_{I}$. We still have $\nim(\Opt(Q))\subseteq\nim(\Opt(P))$
by Proposition~\ref{prop:optStruct}. Let $M$ be an option of $P$
in $X_{I}$. Since $\pty(M)\ne\pty(I)$, $M$ is different from $I$.
So $M$ has an option $R\in X_{I}$. Then $\pty(R)=\pty(I)=\pty(Q)$
and $\lceil R\rceil=I=\lceil Q\rceil$, so $(P,Q)\succ(R,Q)$. Hence
$\nim(R)=\nim(Q)$ by induction. This implies
\[
\mex(\overbrace{\nim(\Opt(Q))}^{A})=\nim(Q)=\nim(R)\in\overbrace{\nim(\Opt(M))}^{B},
\]
 where $A:=\nim(\Opt(Q))$ and $B:=\nim(\Opt(M)$. Thus, by Lemma
\ref{lem:mex},
\[
\mex(\overbrace{\nim(\Opt(Q))}^{A}\cup\{\overbrace{\nim(M)}^{\mex(B)}\})=\mex(\overbrace{\nim(\Opt(Q))}^{A}).
\]
If $N\in\Opt(P)$ then either $N\not\in X_{I}$ or $N\in X_{I}$.
In the first case, Proposition~\ref{prop:optStruct} implies that
$Q$ has an option $L$ with $\lceil L\rceil=\lceil N\rceil$. Since
$\pty(L)=\pty(N)$, $\nim(N)=\nim(L)$ by induction. So $\nim(N)=\nim(L)\in\nim(\Opt(Q))$.
In the second case $\nim(M)=\nim(N)$ by induction since $(P,Q)\succ(M,N)$.
Hence
\begin{align*}
\nim(P) & =\mex(\nim(\Opt(P)))\\
 & =\mex(\nim(\Opt(Q))\cup\{\nim(M)\})\\
 & =\mex(\nim(\Opt(Q))=\nim(Q)
\end{align*}
since $\nim(\Opt(P))=\nim(\Opt(Q))\cup\{\nim(M)\}$.
\end{proof}
\begin{defn}
The \emph{type }of the structure class $X_{I}$ is
\[
\type(X_{I}):=(\pty(I),\nim_{0}(X_{I}),\nim_{1}(X_{I})).
\]
 If $I=S$ then $\nim_{0}(X_{S}):=0$ and $\nim_{1}(X_{S}):=0$. If
$I\ne S$ then
\[
\begin{aligned}\nim_{\pty(I)}(X_{I}) & :=\mex(\nim_{1-\pty(I)}(\Opt(X_{I})),\\
\nim_{1-\pty(I)}(X_{I}) & :=\mex(\nim_{\pty(I)}(\Opt(X_{I}))\cup\{\nim_{\pty(I)}(X_{I})\})
\end{aligned}
\]
are defined recursively. We call the recursive computation of types
using the options of structure classes \emph{type calculus}.
\end{defn}

Note that $X_{S}$ is the only structure class without options. 
\begin{example}
Assume that $\Opt(X_{I})=\{X_{J},X_{K}\}$ with $\type(X_{J})=(0,0,3)$
and $\type(X_{K})=(1,2,0)$. If $\pty(I)=0$ then
\[
\begin{aligned}\nim_{0}(X_{I}) & =\mex(\{\nim_{1}(X_{J}),\nim_{1}(X_{K})\})=\mex(\{3,0\})=1,\\
\nim_{1}(X_{I}) & =\mex(\{\nim_{0}(X_{J}),\nim_{0}(X_{K}),\nim_{0}(X_{I})\})=\mex(\{0,2,1\})=3
\end{aligned}
\]
and $\type(X_{I})=(0,1,3)$. If $\pty(I)=1$ then
\[
\begin{aligned}\nim_{1}(X_{I}) & =\mex(\{\nim_{0}(X_{J}),\nim_{0}(X_{K})\})=\mex(\{0,2\})=1,\\
\nim_{0}(X_{I}) & =\mex(\{\nim_{1}(X_{J}),\nim_{1}(X_{K}),\nim_{1}(X_{I})\})=\mex(\{3,0,1\})=2
\end{aligned}
\]
and $\type(X_{I})=(1,2,1)$.
\end{example}

The type of a structure class $X_{I}$ encodes the parity of $I$
and the nim numbers of the positions in $X_{I}$. 
\begin{prop}
If $P\in X_{I}$ then $\nim(P)=\nim_{\pty(P)}(X_{I})$.
\end{prop}

\begin{proof}
We use structural induction on the positions, together with Propositions~\ref{prop:optStruct}
and \ref{prop:mainThm}. The statement is clearly true for $I=S$.

Any option $Q$ of position $I$ is in $X_{J}$ for some $X_{J}\in\Opt(X_{I})$.
On the other hand, if $X_{J}\in\Opt(X_{I})$ then $X_{J}$ contains
an option $Q$ of $I$. The parity $\pty(Q)=1-\pty(I)$ of $Q$ is
the opposite of the parity of $I$. Hence $\nim(I)=\mex(\nim(\Opt(I)))=\nim_{\pty(I)}(X_{I})$
by induction. 

First assume that $P$ is a position in $X_{I}$ such that $\pty(P)=\pty(I)$.
Then $\nim(P)=\nim(I)$ by Proposition~\ref{prop:mainThm}. Thus
$\nim(P)=\nim(I)=\nim_{\pty(I)}(X_{I})=\nim_{\pty(P)}(X_{I}).$

Now assume that $P$ is a position in $X_{I}$ such that $\pty(P)=1-\pty(I)$.
Then the options of $P$ have parity $\pty(I)$. Since $P$ is strictly
smaller than $I$, $P$ must have an option in $X_{I}$. Every $Q$
in $\Opt(P)\cap X_{I}$ satisfies $\nim(Q)=\nim_{\pty(Q)}(X_{I})=\nim_{\pty(I)}(X_{I})$
by induction. Proposition~\ref{prop:optStruct} implies that every
option of $P$ that is not in $X_{I}$ must be a position in $X_{J}$
for some $X_{J}\in\Opt(X_{I})$. On the other hand, Proposition~\ref{prop:optStruct}
also implies that if $X_{J}\in\Opt(X_{I})$ then $X_{J}$ contains
an option of $P$. Every $Q$ in $\Opt(P)\cap X_{J}$ with $X_{J}\in\Opt(X_{I})$
satisfies $\nim(Q)=\nim_{\pty(Q)}(X_{J})=\nim_{\pty(I)}(X_{J})$ by
induction. Thus
\[
\begin{aligned}\nim(P) & =\mex(\nim(\Opt(P)))\\
 & =\mex(\nim((\Opt(P)\cap\bigcup\Opt(X_{I}))\cup(\Opt(P)\cap X_{I}))\\
 & =\mex(\nim_{\pty(I)}(\Opt(X_{I}))\cup\{\nim_{\pty(I)}(X_{I})\})\\
 & =\nim_{1-\pty(I)}(X_{I})=\nim_{\pty(P)}(X_{I}).
\end{aligned}
\]

\end{proof}
Note that $X_{I}=\{I\}$ is possible. In this case $X_{I}$ contains
no position with parity $1-\pty(I)$. Also note that the nim number
of the game is the nim number of the starting position $\emptyset$.
So $\nim(\GEN(S,W))=\nim(\emptyset)=\mex_{0}(X_{\Phi})$ is the second
component of $\type(X_{\Phi})$.

\subsection{Structure diagrams\label{sec:structureDiagrams}}

The \emph{structure digraph} of $\GEN(S,W)$ has vertex set $\{X_{I}\mid I\in\mathcal{I}\}$
and arrow set $\{(X_{I},X_{J})\mid X_{J}\in\Opt(X_{I})\}$. We visualize
the structure digraph with a \emph{structure diagram} that also shows
the type of each structure class. Within a structure diagram, a vertex
$X_{I}$ is represented by a triangle pointing up or down depending
on the parity of $I$. The triangle points down when $\pty(I)=1$
and points up when $\pty(I)=0$. The numbers within each triangle
represent the nim numbers of the positions within the structure class.
The first number is the common nim number of all even positions in
$X_{I}$, while the second number is the common nim number of all
odd positions in $X_{I}$.

We call an automorphism $\alpha$ of the structure digraph \emph{size
preserving} if $|I|=|J|$ for all $\alpha(X_{I})=X_{J}$. It is often
useful to work with the quotient of the structure digraph with respect
to orbit equivalence determined by the size preserving automorphisms
of the structure digraph. We call the corresponding structure diagram
the \emph{orbit quotient structure diagram}. We draw these quotient
diagrams using representative structure classes. Since orbit equivalent
structure classes have the same type, the nim number of the game can
be determined using the orbit quotient structure diagram. 

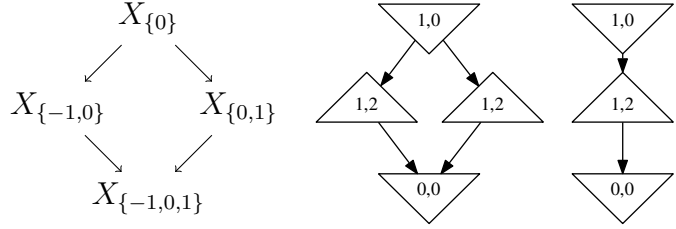
\begin{figure}
\begin{tabular}{ccc}
\begin{tikzpicture}[scale=1.2]
\node (p1) at (0,-1) {$X_{\{0\}}$};
\node (p2) at (-1,-2) {$X_{\{-1,0\}}$};
\node (p3) at (0,-3) {$X_{\{-1,0,1\}}$};
\node (p4) at (1,-2) {$X_{\{0,1\}}$};
\draw[->] (p1) -- (p2);
\draw[->] (p2) -- (p3);
\draw[->] (p1) -- (p4);
\draw[->] (p4) -- (p3);
\end{tikzpicture} & \includegraphics[scale=0.5]{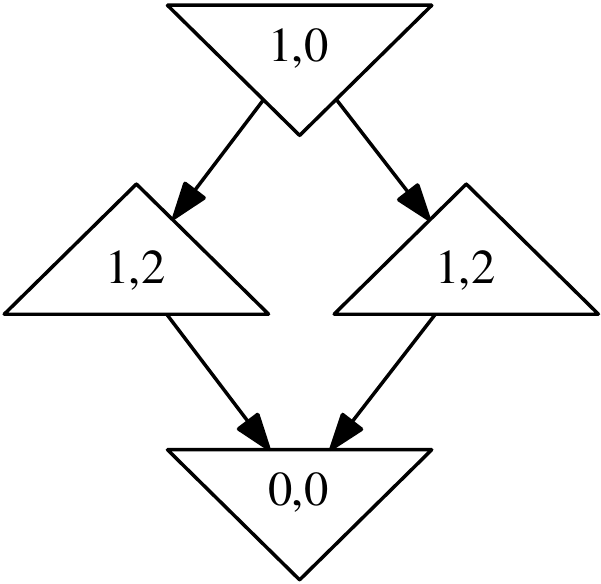} & \includegraphics[scale=0.5]{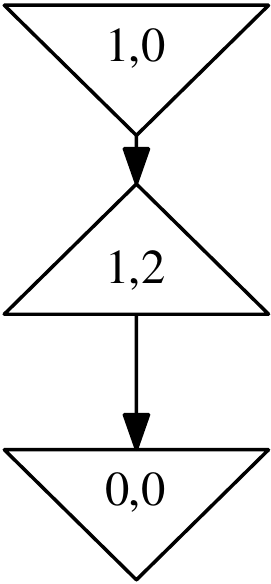}\tabularnewline
\end{tabular}

\caption{\label{fig:3InLine-1}Structure digraph, structure diagram, and orbit
quotient structure diagram of $\protect\GEN(S)$ with $S=\{-1,0,1\}$.}
\end{figure}

\begin{example}
The structure digraph and structure diagrams of $\GEN(S)$ with the
affine convex geometry on $S=\{-1,0,1\}$ considered in Examples~\ref{exa:3InLine}
and \ref{exa:3InLineB} appear in Figure~\ref{fig:3InLine-1}. We
demonstrate type calculus by verifying that $\nim(\GEN(S))=\nim(\emptyset)=\nim_{0}(X_{\{0\}})=1$.
We have 
\[
\begin{aligned}\nim_{1}(X_{\{0\}}) & =\mex(\{\nim_{0}(X_{\{-1,0\}}),\nim_{0}(X_{\{0,1\}})\})=\mex(\{1\})=0,\\
\nim_{0}(X_{\{0\}}) & =\mex(\{\nim_{1}(X_{\{-1,0\}}),\nim_{1}(X_{\{0,1\}})\}\cup\{\nim_{1}(X_{\{0\}})\})\\
 & =\mex(\{2\}\cup\{0\})=1.
\end{aligned}
\]
\end{example}

\begin{rem}
If we make a move in the achievement game on an affine convex geometry,
then the rest of the game is essentially played on a convex geometry
that was created from the original convex geometry by the deletion
of the selected point. Since the selected points do not necessarily
form a convex set, the deletion of these points may result in a convex
geometry in which the empty set is not convex. 
\end{rem}

\begin{figure}
\begin{tabular}{cccc}
\begin{tikzpicture}
\node[label={[label distance=-2mm]180:$3$}] at (-1,0) {$\bullet$};
\node[label={[label distance=-2mm]$2$}] at (0,0) {$\scriptstyle\otimes$};
\node[label={[label distance=-2mm]0:$4$}] at (1,0) {$\bullet$};
\node[label={[label distance=-2mm]$1$}] at (0,1) {$\scriptstyle\otimes$};
\end{tikzpicture} & \begin{tikzpicture}
\node at (-1,0) {$\bullet$};
\node at (0,0) {$\scriptstyle\otimes$};
\node at (1,0) {$\bullet$};
\node at (0,1) {$\circ$};
\end{tikzpicture} & \begin{tikzpicture}
\node at (-1,0) {$\bullet$};
\node at (0,0) {$\circ$};
\node at (1,0) {$\bullet$};
\node at (0,1) {$\scriptstyle\otimes$};
\end{tikzpicture} & \begin{tikzpicture}
\node at (-1,0) {$\circ$};
\node at (0,0) {$\scriptstyle\otimes$};
\node at (1,0) {$\bullet$};
\node at (0,1) {$\scriptstyle\otimes$};
\end{tikzpicture}\tabularnewline
$\{\{1,3\},\{1,4\},\{2,3,4\}\}$ & $\{\{3\},\{4\}\}$ & $\{\{3,4\}\}$ & $\{\{1\},\{2,4\}\}$\tabularnewline
 &  &  & \tabularnewline
\includegraphics[scale=0.5]{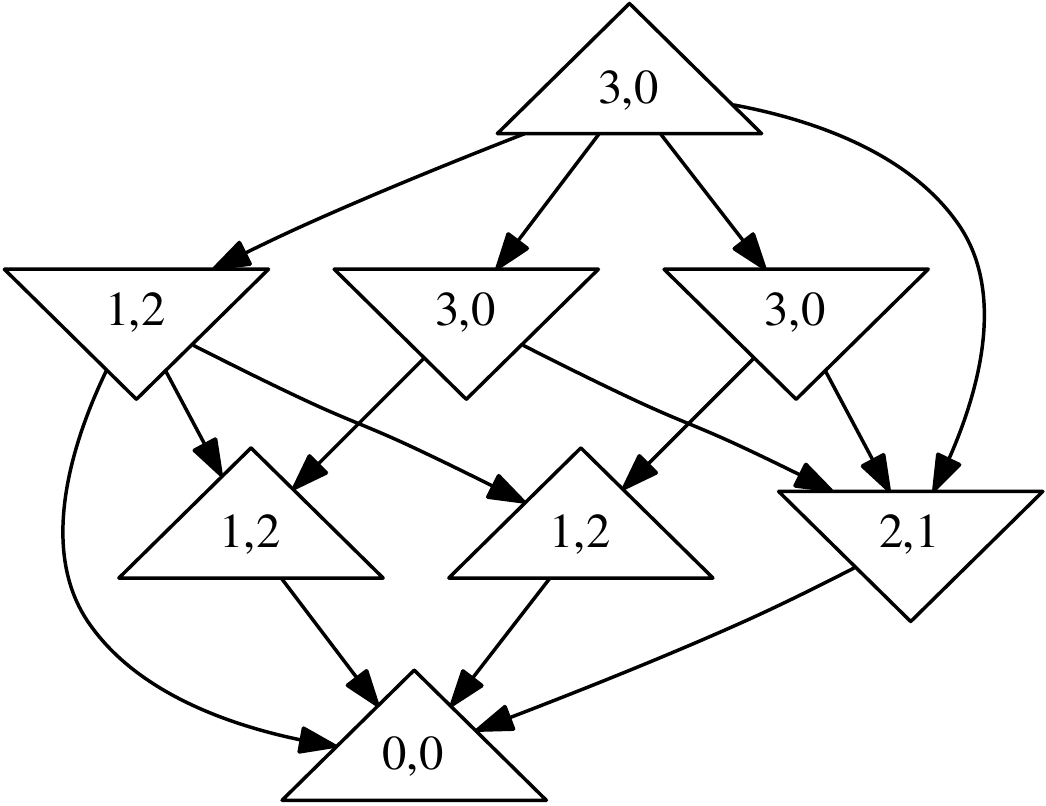} & \includegraphics[scale=0.5]{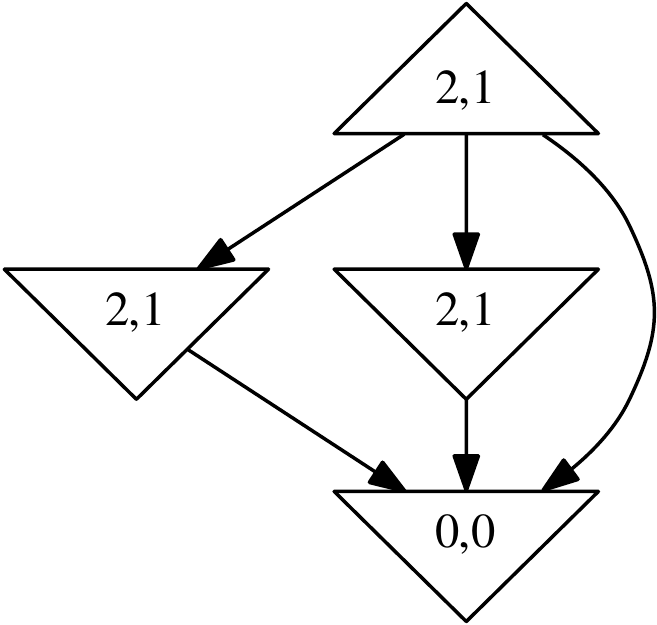} & \includegraphics[scale=0.5]{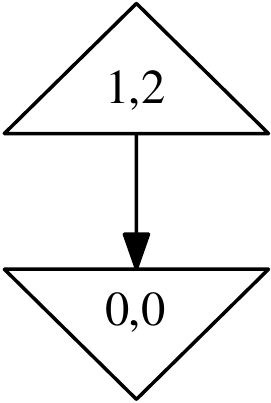} & \includegraphics[scale=0.5]{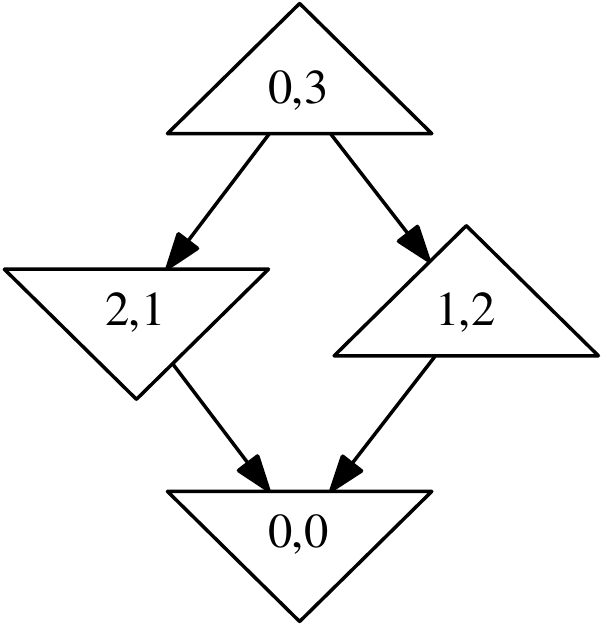}\tabularnewline
\end{tabular}

\caption{\label{fig:deleted}Structure diagram of the achievement game on an
affine convex geometry and some of its deleted convex geometries.
The second row contains the maximally non-generating sets. Deleted
points are shown with $\circ$ and the elements of the goal set are
shown with ${\scriptstyle \otimes}$. }
\end{figure}
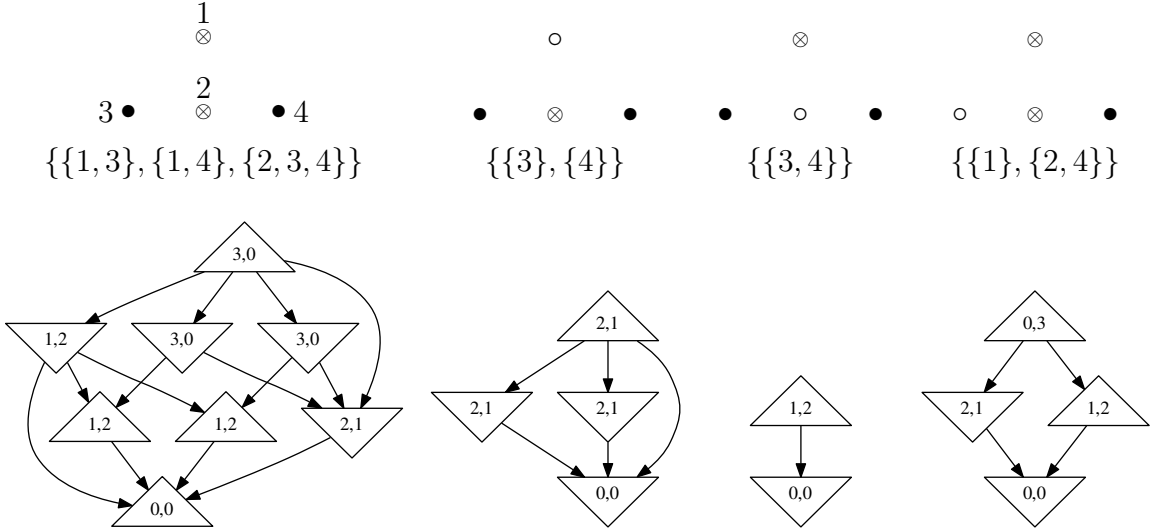

\begin{example}
Figure~\ref{fig:deleted} shows the structure diagrams of the achievement
game on an affine convex geometry and the effect of deletion on the
structure diagrams. Note that deletion of a point produces a structure
diagram that is a sub-diagram of the original diagram. This sub-diagram
starts at an option of $X_{\Phi}$ and it has reversed parities and
nim numbers.
\end{example}

\section{Winning subsets containing only extreme points \label{sec:extremeCharacterizations}}

In this section we characterize the nim number of $\GEN(S,W)$ where
$W\subseteq\Ex(S)$.
\begin{prop}
\label{prop:maxConvex}If $W\subseteq\Ex(S)$ then the set $\mathcal{M}$
of maximally non-generating subsets of $\GEN(S,W)$ is $\{M_{v}\mid v\in W\}$,
where $M_{v}:=S\setminus\{v\}$.
\end{prop}

\begin{proof}
We show that $M_{v}$ for $v\in W$ is a maximally non-generating
set. Since $v$ is an extreme point of $S$, $v\not\in\tau(M_{v})$.
So $M_{v}$ is a non-generating set. If $M_{v}\subset N\subseteq S$
then $N=S$ is a generating set. So $M_{v}$ is maximally non-generating.

We show that every maximally non-generating set is $M_{v}$ for some
$v\in W$. Let $M$ be a maximally non-generating set. Since $M$
is non-generating, there must be a $v$ in $W\setminus M$. We clearly
have $M\subseteq M_{v}$, which means $M_{v}$ a non-generating superset
of $M$. Thus $M=M_{v}$ since $M$ is maximally non-generating.
\end{proof}
\begin{cor}
\label{cor:intersectSub}If $W\subseteq\Ex(S)$ then the set of intersection
subsets of $\GEN(S,W)$ is $\mathcal{I}=\{S\setminus V\mid V\subseteq W\}$. 
\end{cor}

\begin{prop}
If $W=\{w\}\subseteq\Ex(S)$, then
\[
\nim(\GEN(S,W))=\begin{cases}
1, & \pty(S)=1\\
2, & \pty(S)=0.
\end{cases}
\]
\end{prop}

\begin{proof}
We have $\mathcal{M}=\{M\}$ and $\mathcal{I}=\{M,S\}$, where $M=S\setminus\{w\}$.
If $|S|$ is odd, then $\nim(\GEN(S,W))=1$ since $M$ is even and
the structure diagram is the one shown in Figure~\ref{fig:diagGenSW}(c).
If $|S|$ is even, then $\nim(\GEN(S))=2$ since $M$ is odd and the
structure diagram is the one shown in Figure~\ref{fig:diagGenSW}(a).
\end{proof}
\begin{figure}
\begin{tabular}{ccccccc}
\includegraphics[scale=0.5]{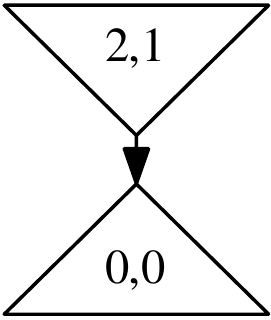} &  & \includegraphics[scale=0.5]{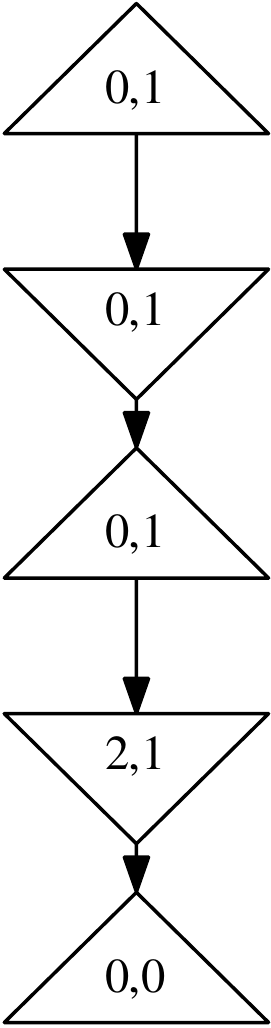} &  & \includegraphics[scale=0.5]{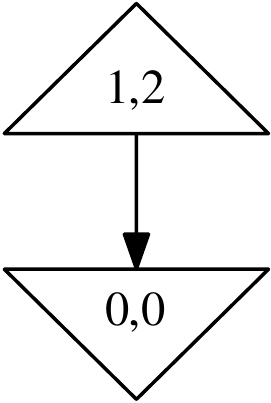} &  & \includegraphics[scale=0.5]{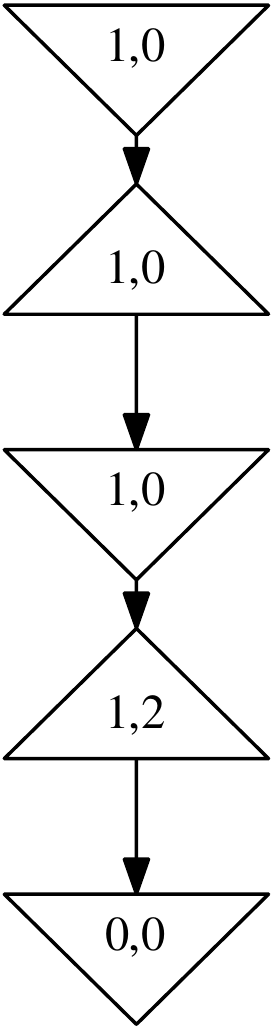}\tabularnewline
\multicolumn{3}{c}{$|S|$ is even} & $\quad$ & \multicolumn{3}{c}{$|S|$ is odd}\tabularnewline
$|W|=1$ &  & $|W|=4$ &  & $|W|=1$ &  & $|W|=4$\tabularnewline
(a) &  & (b) &  & (c) &  & (d)\tabularnewline
\end{tabular}

\caption{\label{fig:diagGenSW}Orbit quotient structure diagrams for $\protect\GEN(S,W)$,
where $W\subseteq\protect\Ex(S)$.}
\end{figure}

\begin{defn}
Let $W\subseteq\Ex(S)$. For $I\in\mathcal{I}$ we define $\delta(I):=|S\setminus I|$
to be the number of goal points missing from $I$.
\end{defn}

\begin{prop}
If $W\subseteq\Ex(S)$ with $|W|\ge2$, then
\[
\nim(\GEN(S,W))=\begin{cases}
0, & \pty(S)=0\\
1, & \pty(S)=1.
\end{cases}
\]
\end{prop}

\begin{proof}
First consider the case when $|S|$ is even. We are going to use type
calculus and structural induction on the elements of $\mathcal{I}$
to show that
\[
\type(X_{I})=\begin{cases}
(0,0,0), & \delta(I)=0\\
(1,2,1), & \delta(I)=1\\
(0,0,1), & \delta(I)>1\text{ and }\delta(I)\equiv_{2}0\\
(1,0,1), & \delta(I)>1\text{ and }\delta(I)\equiv_{2}1
\end{cases}
\]
for all $I\in\mathcal{I}$, as shown in Figure~\ref{fig:diagGenSW}(b).
Then we will have
\[
\nim(\GEN(S,W))=\nim(\emptyset)=\nim_{0}(X_{\Phi})=0
\]
because $\delta(\Phi)=|W|\geq2$. If $\delta(I)=0$ then $X_{I}=X_{S}$
and $\type(X_{S})=(0,0,0)$. If $\delta(I)=1$ then $I=S\setminus\{w\}$
for some $w\in W$, so $\pty(I)=1-\pty(S)=1$. Then $\nim(I)=\mex(\nim(\Opt(I)))=\mex(\nim(S))=\mex\{0\}=1$.
Furthermore, for any $P\in X_{I}$ satisfying $\pty(P)=0$, we have
$\nim(P)=\mex(\nim(\Opt(P)))=\mex\{0,1\}=2$. Hence $\type(X_{I})=(1,2,1)$.

Suppose $\delta(I)\ge2$. We have two cases to consider. If $\pty(I)=0$
then $\delta(I)=|S\setminus I|\equiv_{2}0$, and so $\delta(J)=|S\setminus J|\equiv_{2}1$
with $\type(J)\in\{(1,2,1),(1,0,1)\}$ for all $J\in\Opt(I)$ by induction.
Hence
\begin{align*}
\type(X_{I}) & =(0,\nim_{0}(X_{I}),\nim_{1}(X_{I}))\\
 & =(0,\mex(\nim_{1}(\Opt(I))),\mex(\nim_{0}(\Opt(I)))\cup\{\nim_{0}(X_{I})\})\\
 & =(0,\mex(\{1\}),\mex(A))=(0,0,1),
\end{align*}
where $A=\{0,2\}$ or $A=\{0\}$.

If $\pty(I)=1$ then $\delta(I)\equiv_{2}1$, and so $\delta(J)=|S\setminus J|\equiv_{2}0$
with $\type(J)=(0,0,1)$ for all $J\in\Opt(I)$ by induction. Hence
\begin{align*}
\type(X_{I}) & =(1,\nim_{0}(X_{I}),\nim_{1}(X_{I}))\\
 & =(1,\mex(\nim_{1}(\Opt(I))\cup\{\nim_{1}(X_{I})\}),\mex(\nim_{0}(\Opt(I))))\\
 & =(1,\mex(\{1\}),\mex(\{0\}))=(1,0,1).
\end{align*}

Now consider the case when $|S|$ is odd. One can show that
\[
\type(X_{I})=\begin{cases}
(1,0,0), & \delta(I)=0\\
(0,1,2), & \delta(I)=1\\
(1,1,0), & \delta(I)>1\text{ and }\delta(I)\equiv_{2}1\\
(0,1,0), & \delta(I)>1\text{ and }\delta(I)\equiv_{2}0
\end{cases}
\]
for all $I\in\mathcal{I}$, as show in Figure~\ref{fig:diagGenSW}(d).
The argument is essentially the same as the one we used in the previous
case but with reversed parities. Hence
\[
\nim(\GEN(S,W))=\nim(\emptyset)=\nim_{0}(X_{\Phi})=1
\]
because $\delta(\Phi)=|W|\geq2$. 
\end{proof}
The essence of the previous proof is captured in Figures~\ref{fig:diagGenSW}(b)
and (d). The parity of $|S|$ determines the direction of the triangles,
while $|W|$ determines the height of the diagram. 

\section{Convex geometries from vertices of trees\label{sec:treeCharacterization}}

Consider a tree graph $T$ with vertex set $S$. The vertex sets of
connected subgraphs of $T$ form a convex geometry on $S$. We call
this the \emph{vertex geometry} of $T$. In this section we study
the achievement game $\GEN(S,W)$ on vertex geometries of trees.

We use the standard $N(v)$ notation for the set of vertices adjacent
to a vertex $v$.
\begin{example}
The vertex geometry of the path graph $P_{3}$ with vertex set $S=\{1,2,3\}$
is $\mathcal{K}=2^{S}\setminus\{\{1,3\}\}$. 
\end{example}

\begin{figure}
\begin{centering}
\begin{tabular}{cc}
\begin{tabular}{c}
\includegraphics{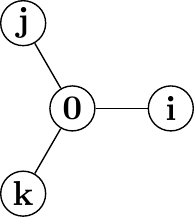}\tabularnewline
\end{tabular} & %
\begin{tabular}{c}
\includegraphics[scale=1.3]{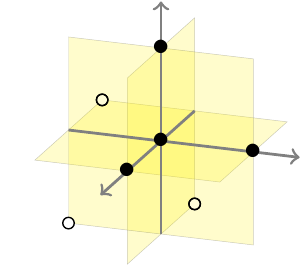}\tabularnewline
\end{tabular}\tabularnewline
\end{tabular}
\par\end{centering}
\caption{\label{fig:K1,3}The complete bipartite graph $K_{1,3}$ and the vertex
geometry of $K_{1,3}$ represented as a deletion of an affine convex
geometry.}
\end{figure}

\begin{example}
Let $R:=\{{\bf 0,i,j,k}\}$ and $D:=\{{\bf -i-j},{\bf -i-k},{\bf -j-k}\}$
be subsets of $\mathbb{R}^{3}$. The vertex geometry of the complete
bipartite graph $K_{1,3}$ can be represented as a deletion of the
affine convex geometry on $S=R\cup D$ by $D$. The graph $K_{1,3}$
and the point set $S$ are shown in Figure~\ref{fig:K1,3}. The set
of convex sets is $\mathcal{K}=2^{R}\setminus\{\{{\bf i},{\bf j}\},\{{\bf j},{\bf k}\},\{{\bf i},{\bf k}\},\{\mathbf{i},\mathbf{j},\mathbf{k}\}\}$.
\end{example}

Removing a vertex $w$ of a tree $T$ creates a forest that we denote
by $T\setminus w$. 
\begin{prop}
Let $(S,\mathcal{K})$ be the vertex geometry of a tree $T$ and $W\subseteq S$.
Then $w\in\Ex(W)$ if and only if the elements of $W\setminus\{w\}$
are vertices in a single connected component of $T\setminus w$.
\end{prop}

\begin{proof}
For each $v\in N(w)$ let $V_{v}$ be the vertex set of the connected
component of $T\setminus w$ containing $v$. The map $v\mapsto V_{v}$
is a bijection from $N(w)$ to the collection of vertex sets of the
connected components of $T\setminus w$. 

First assume $w_{1}$ and $w_{2}$ are elements of $W\setminus\{w\}$
that are vertices in different connected components of $T\setminus w$.
Then $w$ is a vertex of every connected subgraph of $T$ that contains
both $w_{1}$ and $w_{2}$. Hence $w\in\tau(\{w_{1},w_{2}\})\subseteq\tau(W\setminus\{w\})$,
so $w\not\in\Ex(W)$. 

Now assume $W\setminus\{w\}$ is contained in the vertex set of a
single connected component $V_{v}$ of $T\setminus w$. Since $V_{v}$
is convex and contains $W\setminus\{w\}$, $\tau(W\setminus\{w\})\subseteq V_{v}$.
So $w\not\in\tau(W\setminus\{w\})$ since $w\not\in V_{v}$. Thus
$w\in\Ex(W)$.
\end{proof}

\subsection{Winning subsets with more than one vertex}

First we consider the case when $W$ has at least two vertices. If
$w\in\Ex(W)$ then $T\setminus w$ has exactly one component that
contains some elements of $W$. We denote the vertex set of this component
by $M_{w}$. It is clear from the construction that $w\mapsto M_{w}:\Ex(W)\to\mathcal{M}$
is a bijection. We use the notation $V_{w}:=S\setminus M_{w}$. Note
that $\{V_{w}\mid w\in\Ex(W)\}$ is a collection of pairwise disjoint
sets.

There is a bijective correspondence between the subsets of $\Ex(W)$
and the intersection subsets given by $A\mapsto M_{A}:2^{\Ex(W)}\to\mathcal{I}$,
where 
\[
M_{A}:=\bigcap\{M_{w}\mid w\in A\}=S\setminus\bigcup\{V_{w}\mid w\in A\}.
\]
Note that $S=M_{\emptyset}$ and $\Phi=M_{\Ex(W)}$. We define the
\emph{deficiency} of a structure class $M_{A}$ by $\delta(M_{A}):=|A|$.
We also define the \emph{signature} of a structure class $X_{M_{A}}$
to be $\sigma(X_{M_{A}}):=(e,o)$, where $e$ is the size of $\{a\in A\mid\pty(V_{a})=0\}$
and $o$ is the size of $\{a\in A\mid\pty(V_{a})=1\}$. The signature
satisfies $\delta(M_{A})=e+o$. The \emph{signature of the game} is
$\sigma(X_{\Phi})$.
\begin{example}
\label{exa:tree}
\begin{figure}
\begin{tabular}{ccc}
\raisebox{1cm}{\includegraphics{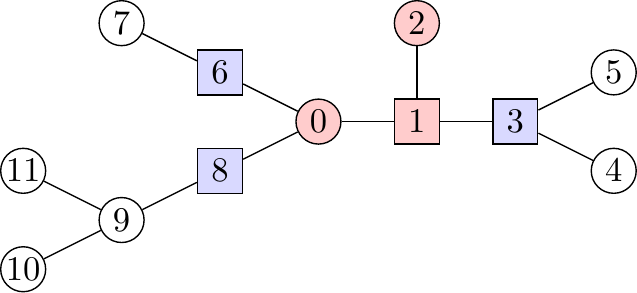}} & \begin{tikzpicture}[yscale=.71]
\node (M) at (0,0) {$M_\emptyset$};
\node (M6) at (-1.5,2) {$M_{\{3\}}$};
\node (M3) at (0,2) {$M_{\{8\}}$};
\node (M8) at (1.5,2) {$M_{\{6\}}$};
\node (M36) at (-1.5,4) {$M_{\{3,8\}}$};
\node (M68) at (0,4) {$M_{\{3,6\}}$};
\node (M38) at (1.5,4) {$M_{\{6,8\}}$};
\node (M368) at (0,6) {$M_{\{3,6,8\}}$};
\draw (M) -- (M6);
\draw (M) -- (M3);
\draw (M) -- (M8);
\draw (M36) -- (M6);
\draw (M36) -- (M3);
\draw (M38) -- (M3);
\draw (M38) -- (M8);
\draw (M68) -- (M6);
\draw (M68) -- (M8);
\draw (M368) -- (M36);
\draw (M368) -- (M68);
\draw (M368) -- (M38);
\end{tikzpicture} & \includegraphics[scale=0.5]{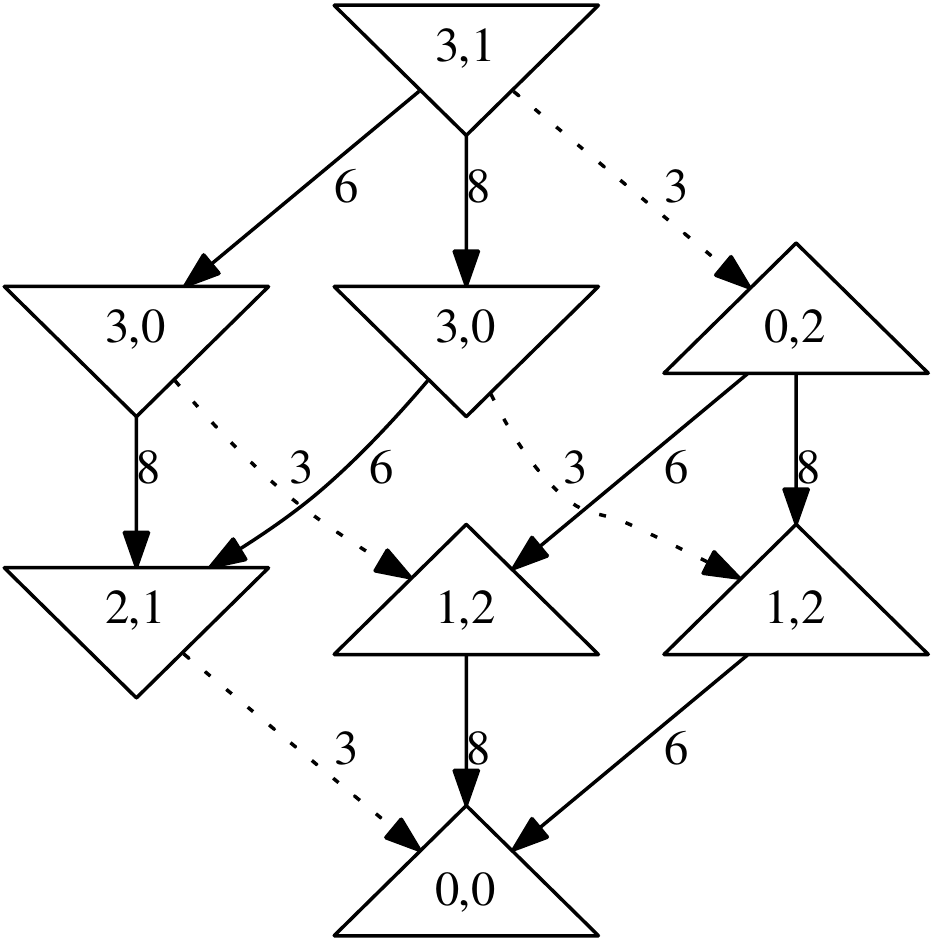}\tabularnewline
\end{tabular}

\caption{\label{fig:tree}Tree graph, lattice of intersection subsets, and
structure diagram of $\protect\GEN(S,W)$ with $W=\{1,3,6,8\}$. Dotted
arrows correspond to a signature change of $(0,-1)$, while solid
arrows correspond to a signature change of $(-1,0)$. The label $a$
on an arrow indicates the change from $M_{A}$ to $M_{A\setminus\{a\}}$.}
\end{figure}
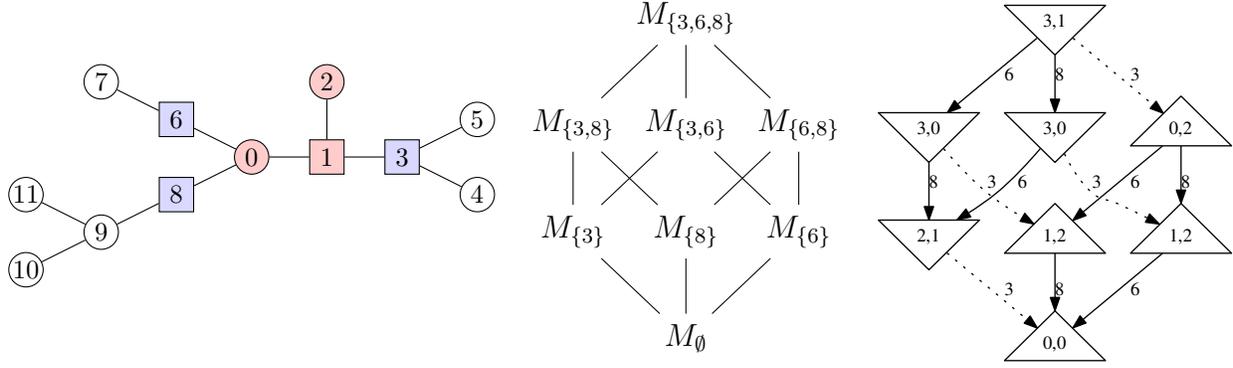
Figure~\ref{fig:tree} shows a tree graph with vertex set $S=\{0,\ldots,11\}$,
the lattice of intersection subsets, and the structure diagram of
$\GEN(S,W)$ with winning set $W=\{1,3,6,8\}$. The set of extreme
points of $W$ is $\Ex(W)=\{3,6,8\}$. The set of maximal non-generating
sets is $\mathcal{M}=\{M_{3},M_{6},M_{8}\}$ with $V_{3}=\{3,4,5\}$,
$V_{6}=\{6,7\}$, and $V_{8}=\{8,9,10,11\}$, so that the Frattini
subset is $\Phi=M_{\{3,6,8\}}=\{0,1,2\}$. Any directed path in the
structure diagram from $X_{M_{A}}$ to $X_{S}=X_{M_{\emptyset}}$
corresponds to the elements of $A$. The signature of $X_{M_{A}}$
can be computed by counting the solid and the dotted arrows along
any such directed path. For example $\sigma(X_{M_{\{3,8\}}})=(1,1)$.
Any directed path from $X_{\Phi}$ to $X_{S}$ contains one dotted
and two solid arrows since the signature of the game is $\sigma(X_{\Phi})=(2,1)$.
\end{example}

\begin{prop}
\label{prop:vertexMore}Let $(S,\mathcal{K})$ be the vertex geometry
of a tree $T$ and $W$ be a subset of $S$ with $|W|\ge2$. If the
signature of the game is $(e,o)$, then
\[
\nim(\GEN(S,W))=\begin{cases}
1, & (e,o)=(1,0)\\
2, & (e,o)\in\{(0,1),(1,2)\}\\
3, & (e,o)\in\{(1,1),(2,1)\}\\
0, & \text{otherwise}
\end{cases}
\]
when $\pty(S)=0$, and
\[
\nim(\GEN(S,W))=\begin{cases}
0, & (e,o)\in\{(0,0),(1,1)\}\\
2, & (e,o)\in\{(1,0),(2,0)\}\\
1, & \text{otherwise}
\end{cases}
\]
when $\pty(S)=1$.
\end{prop}

\begin{figure}
\begin{tabular}{cc}
\includegraphics[scale=0.5]{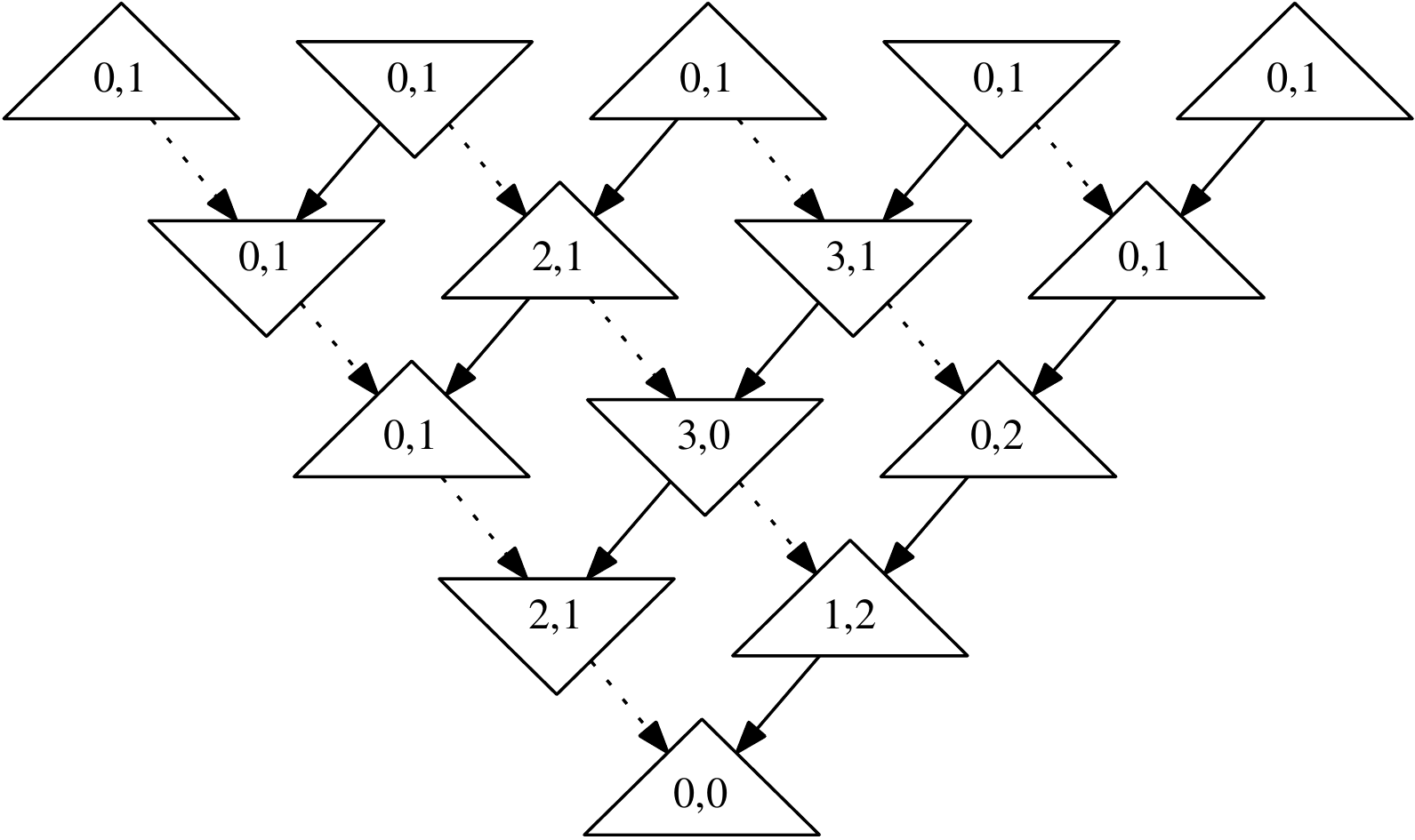} & \includegraphics[scale=0.5]{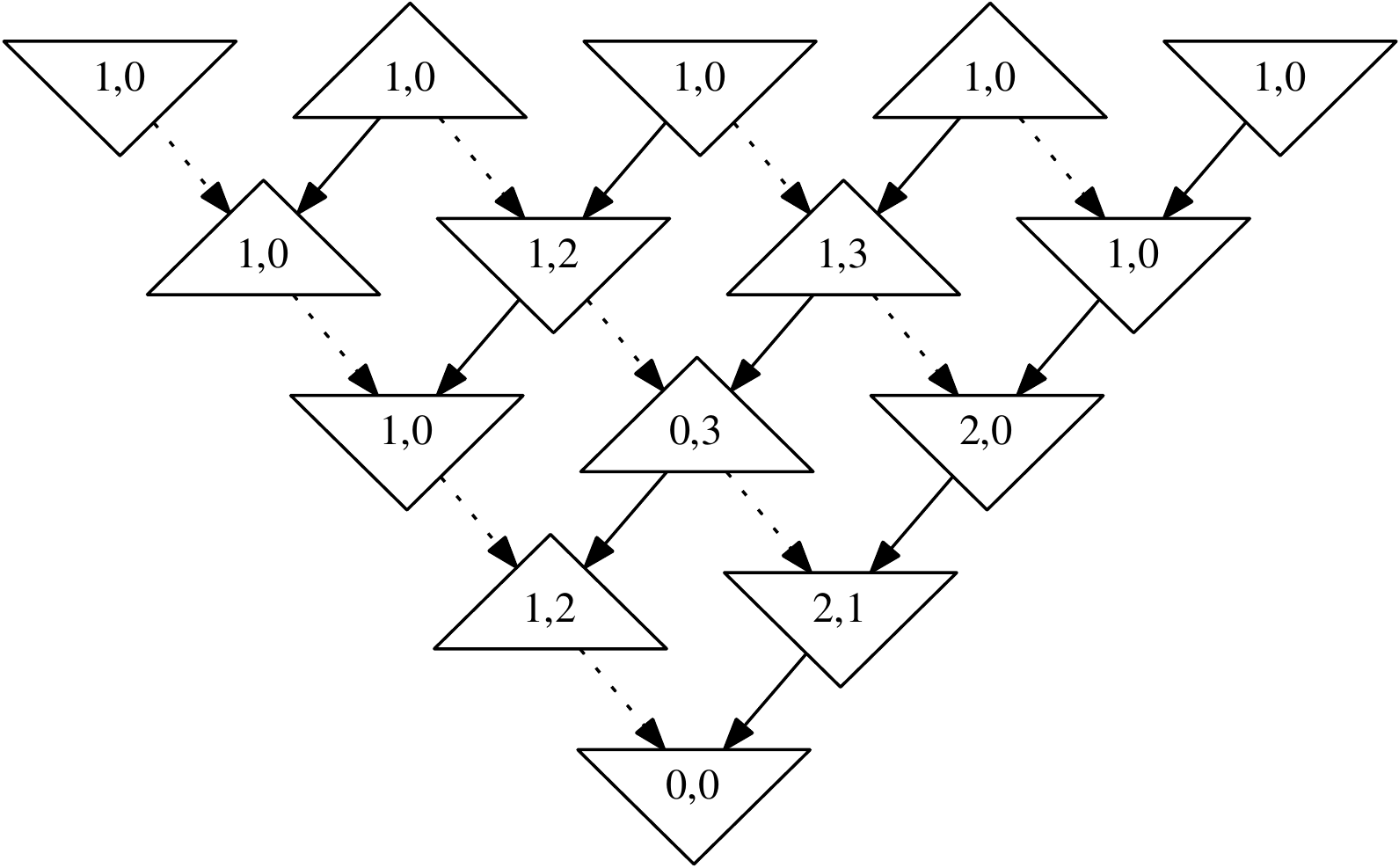}\tabularnewline
\end{tabular}

\caption{\label{fig:structInd}Type calculus computation of the possible structure
class types with deficiency less than 5. A solid arrow represents
no change in parity while a dotted arrow represents a parity change.}
\end{figure}

\begin{proof}
Let $X_{J}$ be an option of $X_{I}$ and $I=M_{A}$. Then $I\cup\{v\}\in X_{J}$
for some $v\in J\setminus I$. So there is a unique $a\in A$ such
that $v\in V_{a}$ and $J=M_{A\setminus\{a\}}$. So every arrow of
the structure diagram corresponds to a unique element $a$ of $\Ex(W)$.
This means $\delta(J)=|A\setminus\{a\}|=|A|-1=\delta(I)-1$. Roughly
speaking this tells us that there are no arrows between structure
classes that are at the same directed distance from $X_{S}$. 

Assume $I=M_{A}$ and $\sigma(X_{I})=(e,o)$. If $e,o\ge1$ then there
are $a,b\in A$ such that $\pty(V_{a})=0$ and $\pty(V_{b})=1$. Hence
\[
\sigma(\Opt(X_{I}))=\begin{cases}
\{(e-1,0)\}, & o=0\\
\{(0,o-1)\}, & e=0\\
\{(e-1,o),(e,o-1)\}, & e,o\ge1.
\end{cases}
\]
The result now follows from type calculus and structural induction
on the structure classes. The details of the type calculus computation
are shown in Figure~\ref{fig:structInd}. The signature of a structure
class is $(e,o)$ where $e$ is the number of solid arrows and $o$
is the number of dotted arrows from the structure class to the bottom
structure class along a directed path. Starting at deficiency $4$
the set of possible types remains fixed by induction.
\end{proof}

\subsection{Winning subsets with only one vertex}

Now we consider the case when $W=\{w\}$ is a singleton set. For each
$v$ in the set $N(w)$ of neighbors of $w$, there is a unique connected
component of $T\setminus w$ that contains $v$. The vertex set $M_{v}$
of this component is a maximal non-generating set. In fact, $\mathcal{M}=\{M_{v}\mid v\in N(w)\}$.
Note that $\mathcal{M}$ is a collection of disjoint subsets. The
\emph{signature} of the game is $(e,o)$, where $e$ is the size of
$\{v\in N(w)\mid\pty(M_{v})=0\}$ and $o$ is the size of $\{v\in N(w)\mid\pty(M_{v})=1\}$. 
\begin{example}
\begin{figure}
\begin{tabular}{ccc}
\raisebox{.3cm}{\includegraphics{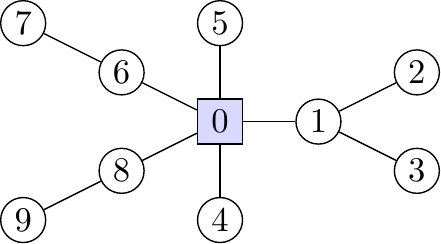}} & ~ & \includegraphics[scale=0.5]{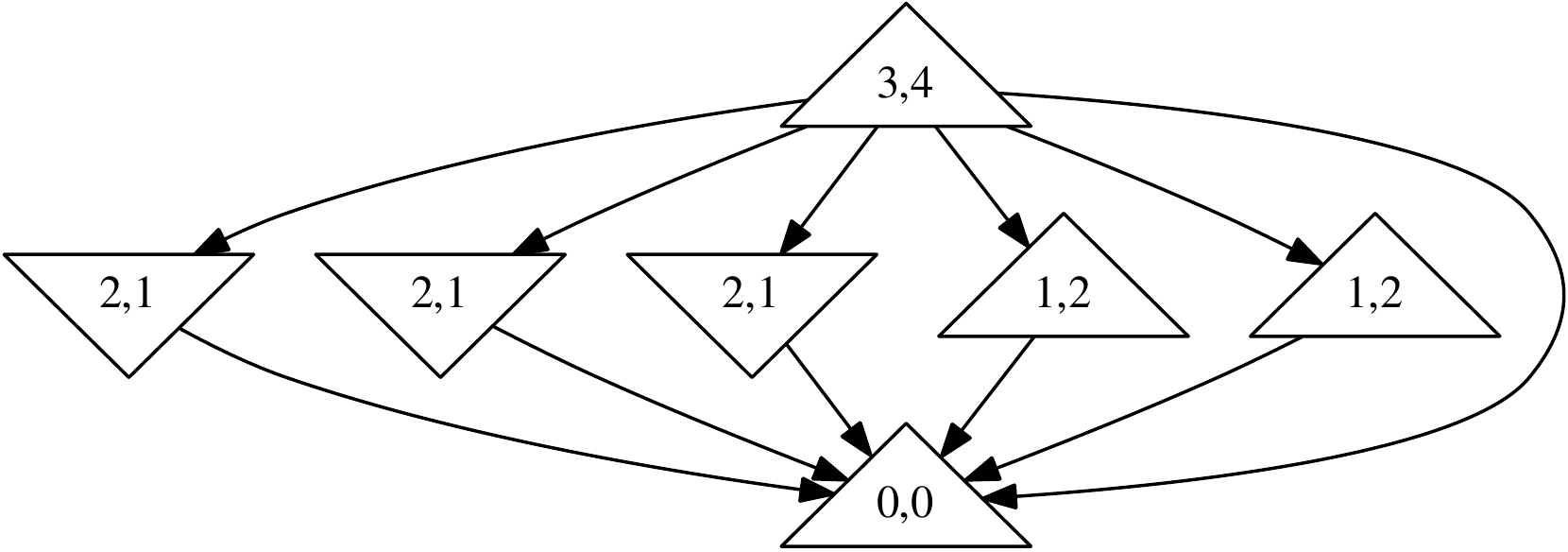}\tabularnewline
\end{tabular}

\caption{\label{fig:tree2}Tree graph and structure diagram of $\protect\GEN(S,W)$
with $W=\{0\}$.}
\end{figure}
Figure~\ref{fig:tree2} shows a tree graph with vertex set $S=\{0,\ldots,9\}$
and the structure diagram of $\GEN(S,W)$ with winning set $W=\{0\}$.
Some of the maximal non-generating sets are $M_{1}=\{1,2,3\}$ and
$M_{5}=\{5\}$, so that the Frattini subset is $\Phi=\emptyset$.
The signature of the game is $(e,o)=(2,3)$.
\end{example}

\begin{figure}
\setlength{\tabcolsep}{10pt}

\begin{tabular}{ccccc}
\includegraphics[scale=0.5]{D1e} & \includegraphics[scale=0.5]{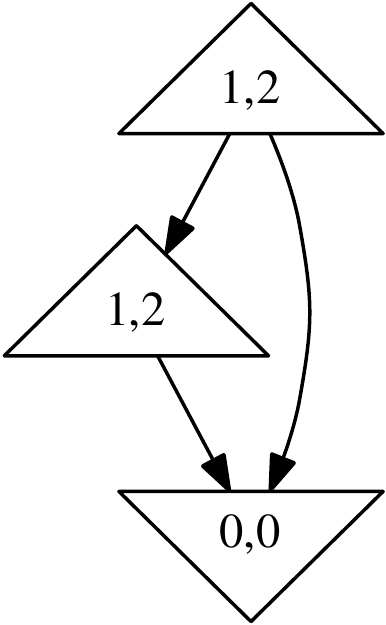} & \includegraphics[scale=0.5]{D1o} & \includegraphics[scale=0.5]{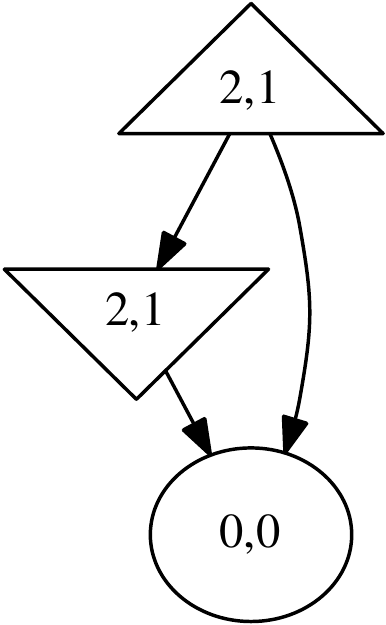} & \includegraphics[scale=0.5]{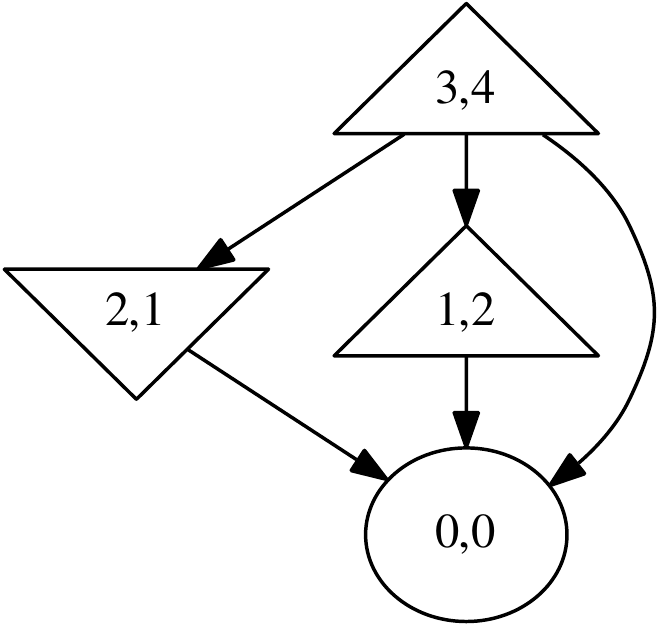}\tabularnewline
$e\le1$ & $e\ge2$ & $e=0$ & $e=0$ & $e\ge1$\tabularnewline
$o=0$ & $o=0$ & $o=1$ & $o\ge2$ & $o\ge1$\tabularnewline
(a) & (b) & (c) & (d) & (e)\tabularnewline
 &  &  &  & \tabularnewline
\end{tabular}

\caption{\label{fig:TreeV1}Orbit quotient structure diagrams for $\protect\GEN(S,\{w\})$
with vertex geometry. The oval shape in subfigures (d) and (e) indicates
that the parity of the structure class could be either $0$ or $1$.}
\end{figure}

\begin{prop}
\label{prop:vertexOne}Let $(S,\mathcal{K})$ be the vertex geometry
of a tree $T$ and $W=\{w\}\subseteq S$. If the signature of the
game is $(e,o)$, then 
\[
\nim(\GEN(S,W))=\begin{cases}
1, & o=0\\
2, & e=0\text{ and }o\ge1\\
3, & e,o\ge1.
\end{cases}
\]
\end{prop}

\begin{proof}
Since $\emptyset\in X_{\Phi}$ and $\{w\}\in X_{S}$, $X_{S}\in\Opt(X_{\Phi})$. 

Assume $o=0$, so that $\pty(S)=1$. If $e=0$ then the nim number
of the game is $1$ as shown in Figure~\ref{fig:TreeV1}(a) since
$S=\{w\}$, $\mathcal{M}=\{\emptyset\}$, and $\mathcal{I}=\{\emptyset,S\}$.
If $e=1$ then $N(w)=\{v\}$, $\mathcal{M}=\{M_{v}\}$, and $\mathcal{I}=\{M_{v},S\}$.
Hence the structure diagram is shown in Figure~\ref{fig:TreeV1}(a)
and the nim number of the game is $1$. If $e\ge2$ then $\Phi=\emptyset$
and $\mathcal{I}=\{\emptyset,S\}\cup\{M_{v}\mid v\in N(w)\}$, so
the orbit quotient structure diagram is shown in Figure~\ref{fig:TreeV1}(b)
and the nim number of the game is $1$.

Assume $e=0$ and $o\ge1$. If $o=1$ then $N(w)=\{v\}$, $\mathcal{M}=\{M_{v}\}$,
and $\mathcal{I}=\{M_{v},S\}$, so the structure diagram is shown
in Figure~\ref{fig:TreeV1}(c) and the nim number of the game is
$2$. If $o\ge2$ then $\Phi=\emptyset$ and $\mathcal{I}=\{\emptyset,S\}\cup\{M_{v}\mid v\in N(w)\}$,
so the orbit quotient structure diagram is shown in Figure~\ref{fig:TreeV1}(d)
and the nim number of the game is $2$. 

Assume $e,o\ge1$, so that $\Phi=\emptyset$ and $\mathcal{I}=\{\emptyset,S\}\cup\{M_{v}\mid v\in N(w)\}$.
So the orbit quotient structure diagram is shown in Figure~\ref{fig:TreeV1}(e)
and the nim number of the game is $3$.
\end{proof}
\begin{cor}
If $(S,\mathcal{K})$ is the vertex geometry of a tree $T$, then
$\nim(\GEN(S,W))\in\{0,1,2,3\}$.
\end{cor}

\section{Affine convex geometries in $\mathbb{R}$\label{sec:affineCharacterization}}

In this section we study the games played on affine convex geometries
in $\mathbb{R}$. It is easy to see that if $S=\{s_{1},\ldots,s_{n}\}\subseteq\mathbb{R}$
then the affine convex geometry on $S$ is isomorphic to the vertex
geometry of a path graph $P_{n}$ with $n$ vertices.
\begin{prop}
Let $S=\{s_{1},\ldots,s_{n}\}\subseteq\mathbb{R}$ such that $s_{i}<s_{i+1}$
for all $i$ and $W=\{s_{i_{1}},\ldots,s_{i_{k}}\}$ such that $i_{1}<i_{2}<\cdots<i_{k}$
and $k\ge2$. Then
\[
\nim(\GEN(S,W))=\begin{cases}
0, & i_{1}\equiv_{2}i_{k}+1\\
1, & i_{1}\equiv_{2}1\equiv_{2}i_{k}\text{ and }n\equiv_{2}1\\
2, & i_{1}\equiv_{2}0\equiv_{2}i_{k}\text{ and }n\equiv_{2}1\\
3, & i_{1}\equiv_{2}i_{k}\text{ and }n\equiv_{2}0.
\end{cases}
\]
\end{prop}

\begin{proof}
We consider the equivalent game on the vertex geometry of the path
graph $P_{n}$ with vertex set $\{1,2,\ldots,n\}$ and $W=\{i_{1},i_{2},\ldots,i_{k}\}$.
The set of extreme points of the winning set is $\Ex(W)=\{i_{1},i_{k}\}$.
We have $|V_{i_{1}}|=i_{1}$ and $|V_{i_{k}}|=n-i_{k}+1$. So the
signature of this game is
\[
\sigma(X_{\Phi})=\begin{cases}
(1,1), & i_{1}\equiv_{2}i_{k}+1\text{ and }n\equiv_{2}1\\
(2,0), & i_{1}\equiv_{2}0\equiv_{2}i_{k}+1\text{ and }n\equiv_{2}0\\
(0,2), & i_{1}\equiv_{2}1\equiv_{2}i_{k}+1\text{ and }n\equiv_{2}0\\
(0,2), & i_{1}\equiv_{2}1\equiv_{2}i_{k}\text{ and }n\equiv_{2}1\\
(2,0), & i_{1}\equiv_{2}0\equiv_{2}i_{k}\text{ and }n\equiv_{2}1\\
(1,1), & i_{1}\equiv_{2}i_{k}\text{ and }n\equiv_{2}0.
\end{cases}
\]
The result now follows from Proposition~\ref{prop:vertexMore}.
\end{proof}
\begin{prop}
Let $S=\{s_{1},\ldots,s_{n}\}\subseteq\mathbb{R}$ such that $s_{i}<s_{i+1}$
for all $i$ and $W=\{s_{k}\}$. Then
\[
\nim(\GEN(S,W))=\begin{cases}
1, & n\equiv_{2}1\equiv_{2}k\\
2, & n\equiv_{2}1\equiv_{2}k+1\\
2, & n\equiv_{2}0\text{ and }k\in\{1,n\}\\
3, & n\equiv_{2}0\text{ and }k\in\{2,\ldots,n-1\}.
\end{cases}
\]
\end{prop}

\begin{proof}
We consider the equivalent game on the vertex geometry of the path
graph $P_{n}$ with vertex set $V=\{1,2,\ldots,n\}$ and $W=\{k\}$.
The neighborhood of $k$ is $N(k)=\{k-1,k+1\}\cap V$. So the signature
of this game is
\[
\sigma(X_{\Phi})=\begin{cases}
(0,0), & n=1=k\\
(1,0), & n\equiv_{2}1\equiv_{2}k\text{ and }k\in\{1,n\}\\
(2,0), & n\equiv_{2}1\equiv_{2}k\text{ and }k\in\{2,\ldots,n-1\}\\
(0,2), & n\equiv_{2}1\equiv_{2}k+1\\
(0,1), & n\equiv_{2}0\text{ and }k\in\{1,n\}\\
(1,1), & n\equiv_{2}0\text{ and }k\in\{2,\ldots,n-1\}.
\end{cases}
\]
The result now follows from Proposition~\ref{prop:vertexOne}.
\end{proof}
Deleted affine convex geometries can be solved using the techniques
of this section.
\begin{example}
Consider the deleted affine convex geometry on the set $S=\{0,1,\ldots,6\}$
with deleted set $D=\{3,5\}$. The achievement game with winning set
$W=\{2,6\}$ is equivalent to $\GEN(S\setminus D,W)$, so it has nim
number $1$. The achievement game with winning set $W=\{1,2\}$ is
equivalent to $\GEN(\{1,2,4,6\},\{1\})$, so it has nim number $2$.
\end{example}

\section{\label{sec:FurtherQuestions}Further directions}

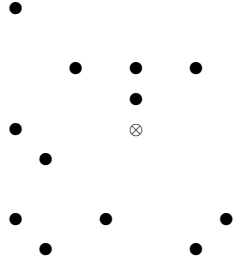
\begin{figure}
\begin{tikzpicture}[scale=.4]
\node at (-4,-3) {$\bullet$};
\node at (-4,0) {$\bullet$};
\node at (-4,4) {$\bullet$};

\node at (-3,-4) {$\bullet$};
\node at (-3,-1) {$\bullet$};
\node at (-2,2) {$\bullet$};

\node at (-1,-3) {$\bullet$};
\node at (0,0) {$\scriptscriptstyle\otimes$};
\node at (0,1) {$\bullet$};

\node at (0,2) {$\bullet$};
\node at (2,-4) {$\bullet$};
\node at (2,2) {$\bullet$};
\node at (3,-3) {$\bullet$};
\end{tikzpicture}

\caption{\label{fig:nim6}Affine convex geometry with $\protect\nim(\protect\GEN(S,\{w\}))=6$.
The winning point $w$ is indicated by $\otimes$.}

\end{figure}

We provide some comments and propose some questions for further study.
\begin{enumerate}
\item The conjecture in Remark~\ref{rem:repConj} probably can be proved
by adjusting the approach of \cite{affineRep}. 
\item Our results may suggest that the nim number of $\GEN(S,W)$ is in
the set $\{0,1,2,3\}$. The point set in Figure~\ref{fig:nim6} provides
an example in $\mathbb{R}^{2}$ where the nim number is $6$. The
companion web site \cite{WEB} provides an example in $\mathbb{R}^{3}$
where the nim number is $8$. What are the possible nim numbers for
convex geometries? What are the possible nim numbers for affine convex
geometries? Does the answer depend on the dimension of the space?
\item The definition of vertex geometry can be generalized to forest graphs.
Our results might generalize to this setting. 
\item Consider a tree graph $T$ with edge set $S$. The edge sets of connected
subgraphs of $T$ form a convex geometry on $S$. What is the nim
number of $\GEN(S,W)$ played on these \emph{edge geometries}?
\item There are several ways to build a convex geometry from a partially
ordered set. What can we say about $\GEN(S,W)$ played on these convex
geometries?
\item What can we say about $\GEN(S,W)$, where $W$ is the Frattini set
of $\GEN(S)$?
\item The original games introduced by \cite{anderson.harary:achievement}
and played on a group can be generalized by allowing a winning set
that is a subset of the group. This may produce interesting results
for special subgroups as the winning set.
\item It might be easier to determine the possible nim values of the avoidance
game $\DNG(S,W)$ played on convex geometries. The avoidance version
was also introduced by \cite{anderson.harary:achievement}. In this
game it is not allowed to select an element that creates a set whose
convex closure contains the winning set. In the group version of the
game, Lagrange's Theorem is a powerful restriction on the possible
structure diagrams \cite{BESspectrumArxive}. Without this restriction,
the spectrum of nim numbers in the convex geometry version could be
all nonnegative integers.
\item Does the Tutte polynomial of the anti-matroid associated with a convex
geometry have any information about $\GEN(S,W)$? Some connection
is expected since both deletion and contraction on a convex geometry
have meaningful game theoretic interpretations.
\end{enumerate}

\section*{}

\bibliographystyle{plain}
\bibliography{game}

\end{document}